\PassOptionsToPackage{sort&compress}{natbib}
 \documentclass[final,3p,times]{elsarticle}

\usepackage[english]{babel}
\usepackage{jcomp}
\usepackage{framed,multirow}
\usepackage[colorlinks=true]{hyperref}

\usepackage{algpseudocode}
\usepackage{amssymb,amsthm,amsmath,bm,amsfonts}

\DeclareMathAlphabet{\mathbbm}{U}{bbm}{m}{n}

\usepackage{dsfont}
\usepackage{latexsym}
\usepackage{subcaption}
\usepackage{url}
\usepackage{xcolor}
\usepackage{mathtools}
\usepackage{cleveref}
\usepackage{showkeys}
\usepackage{multirow}
\usepackage{booktabs}
\usepackage{enumitem}  
\usepackage{array}
\usepackage{cancel}
\usepackage{algorithm}
\usepackage{algpseudocode}
\usepackage{mathabx} 

\journal{Paper Submitted to Comm. Comput. Phys.}

\newtheorem{theorem}{Theorem}[section]

\newtheorem{definition}[theorem]{Definition}

\numberwithin{equation}{section}

\theoremstyle{definition}

\theoremstyle{remark}

\begin{document}

\begin{frontmatter}
	
       \title {{\bf Asymptotic preserving methods for the low mach limit in discrete velocity models approximating kinetic equations}\tnoteref{tnote1} 
	}

	\tnotetext[tnote1]{The work of Theresa K\"{o}fler has been supported by the European Union's Horizon Europe research and innovation programme under grant agreement No. 101086214 (DATAHYKING project).
	The work of Axel Klar has been supported by German Research Foundation (DFG) through grant  KL-1105-34,  ``Asymptotic preserving higher-order  meshfree schemes for kinetic equations".
	The work of Lorenzo Pareschi was partially supported by the Royal Society under the Wolfson Fellowship ``Uncertainty quantification, data-driven simulations and learning of multiscale complex systems governed by PDEs" and by the FIS2023-01334 Advanced Grant ``Tackling complexity: advanced numerical approaches for multiscale systems with uncertainties" (ADAMUS). The work of Giacomo Dimarco has been supported by by the Italian Ministry of University and Research (MUR) through the PRIN 2022 project (No. 2022KKJP4X) ``Advanced numerical methods for time dependent parametric partial differential equations with applications". This work has been written within the activities of GNFM and GNCS groups of INdAM (Italian National Institute of High Mathematics). Program codes can be found on \url{https://gitlab.rhrk.uni-kl.de/theresa.koefler/ap_low_mach}.
	}
	
	\author[FE,CMCS]{Giacomo \snm{Dimarco}}
	\ead{giacomo.dimarco@unife.it}

	\author[KA]{Axel \snm{Klar}}
	\ead{klar@mathematik.uni-kl.de}

	\author[FE,KA]{Theresa \snm{K\"{o}fler}}
	\ead{theresa.kofler@unife.it}

	\author[FE,HW]{Lorenzo \snm{Pareschi}}
	\ead{l.pareschi@hw.ac.uk}

%	\address[KA]{Department of Mathematics, RPTU Kaiserslautern-Landau, Kaiserslautern}
%
%	\address[FE]{Department of Mathematics and Computer Science, University of Ferrara, Italy.}
%	\address[CMCS]{Center for Modeling, Computing and Statistics, University of Ferrara, Italy.}
%	\address[HW]{Maxwell Institute for Mathematical Sciences and Department of Mathematics, School of Mathematical and Computer Sciences, Heriot-Watt University, Edinburgh, UK.}

	\address[KA]{Department of Mathematics, RPTU Kaiserslautern-Landau, Kaiserslautern, 67663, RLP, Germany}

	\address[FE]{Department of Mathematics and Computer Science, University of Ferrara, Ferrara, 44121, FE, Italy.}
	\address[CMCS]{Center for Modeling, Computing and Statistics, University of Ferrara, Ferrara, 44121, FE, Italy.}
	\address[HW]{Maxwell Institute for Mathematical Sciences and Department of Mathematics, School of Mathematical and Computer Sciences, Heriot-Watt University, Edinburgh, EH14 4AS, SCO, UK.}

\begin{abstract} 
We consider a Lattice Boltzmann–type discrete velocity model in the low Mach number scaling and develop a corresponding numerical scheme that remains uniformly valid across all regimes of the mean free path, from the kinetic to the hydrodynamic scale. The proposed framework ensures high-order temporal accuracy through the use of Implicit-Explicit Runge–Kutta methods, which provide stability and efficiency in stiff regimes, while spatial resolution is enhanced by combining finite-difference WENO reconstructions with high-order central difference approximations. In the appropriate asymptotic limit, the scheme reduces to a high-order finite-difference formulation of the incompressible Navier–Stokes equations, thereby guaranteeing physical consistency of the numerical approximation with the limit model. To corroborate the theoretical findings, a set of numerical experiments is performed on two-dimensional benchmark problems, which confirm the accuracy, stability, and versatility of the method across different flow regimes.

\end{abstract}
	
	
\end{frontmatter}

%\tableofcontents

\section{Introduction}
Kinetic equations provide a natural framework for linking microscopic particle dynamics with macroscopic fluid descriptions \cite{Chap,Cer,Levermore91}. In particular, when both the Knudsen and Mach numbers are small, from such models, it is possible to asymptotically recover the incompressible Navier–Stokes (INS) equations \cite{Levermore91,Esposito}. From a numerical perspective, Discrete Velocity Methods (DVMs) offer a systematic strategy to approximate the Boltzmann equation and related kinetic models by replacing the continuous velocity space with a finite set of discrete velocities \cite{Dimarco_acta,Illner,Mieussens}. This leads to a system of coupled, through the collision operator, partial differential equations, each describing the transport of a particle population along one discrete direction. The accuracy of the method then depends on the choice of the discrete velocity set, on the treatment of the collision operator, which must guarantee the preservation of the conservation laws and assure consistency with macroscopic dynamics and on the choice of suitable time and space discretization strategies \cite{Dimarco_acta,Sterling}.

Among DVMs, the Lattice Boltzmann Method (LBM) has become particularly prominent in the last decades \cite{Li-Shi,ChenDoolen1998,BINZI1992145,benzi1992lattice,Inamuro}. In this framework, discrete velocities are constrained to lie on a regular lattice, enabling exact streaming of particle populations between neighboring nodes. This property ensures high computational efficiency, making LBM well suited for large-scale simulations. Despite the low number of discrete velocity used, through the appropriate design of equilibrium distributions and collision operators, LBMs are able to recover the Navier–Stokes equations in the hydrodynamic limit while retaining certain kinetic features that extend their applicability beyond the continuum regime. Compared to general DVMs, which may adopt more flexible velocity sets to handle rarefied or strongly non-equilibrium flows, LBMs are optimized for robustness and efficiency in the low Mach and moderate Knudsen number regimes \cite{benzi1992lattice}. 

From a different perspective, LBM and related formulations can also be interpreted as relaxation systems for the incompressible Navier–Stokes equations \cite{Mapundi}. From a numerical standpoint, relaxation-based schemes have been widely used both for discretizing such systems as well as for constructing methods for kinetic equations with stiff relaxation terms, in regimes ranging from fluid-dynamic to diffusive \cite{JinLevermore1996,Russo,JinPareschiToscani1998,Klar1998,JinPareschiToscani1999,NaldiPareschiToscani2002,Klar99,K992}. These approaches typically belong to the class of asymptotic-preserving (AP) methods, whose main objective is to ensure stability and accuracy uniformly across different asymptotic regimes, without the need to refine the discretization parameters according to the scale of the problem. In particular, AP schemes are designed so that, as the scaling parameter tends to zero, the numerical method automatically degenerates into a consistent discretization of the limiting macroscopic equations. This property makes them particularly effective in multiscale problems, where kinetic and fluid regimes coexist and interact \cite{DimarcoAP1,DimarcoAP2,KlarAP}.

In this context, Implicit–Explicit (IMEX) time integrator methods have emerged as a powerful tool for the time discretization of partial differential equations with stiff relaxation terms \cite{IMEX_PR,Bosc,Bosc1,Carp,AscherRuuthSpiteri1997,BoscarinoPareschiRusso2024,Ruuth,Multis,Bosc2}. These schemes typically treat stiff collision operators implicitly while handling transport terms explicitly, thus combining stability with computational efficiency. Unlike fully implicit methods, which are often too costly in high dimensions, IMEX schemes avoid the need to solve large nonlinear systems, while still permitting time steps that are not restricted by the small relaxation parameter. Moreover, when properly designed, IMEX methods inherit the asymptotic-preserving property, guaranteeing a smooth transition from the kinetic to the hydrodynamic regime without loss of accuracy. Their flexibility and high-order accuracy make them particularly well suited for coupling with advanced spatial discretizations, such as WENO \cite{Shu} or central difference schemes, in order to achieve robust and accurate solvers across a wide range of flow regimes.

In the standard LBM formulation, the discretization proceeds by introducing a spatial mesh aligned with the discrete velocities, yielding an explicit finite-difference scheme for a discrete-velocity Boltzmann equation with a BGK-type collision operator (cf. \cite{JUNK}). However, the explicit nature of the method limits its efficiency in the low Mach regime, where stiffness requires very fine meshes in both space and time to recover the incompressible limit equations accurately, leading to high computational costs, cf. \cite{DegondTang2011,DimarcoLoubereVignal2017}. To overcome these limitations, implicit or semi-implicit strategies have been proposed in the recent past for kinetic equations with stiff relaxation terms, producing AP schemes that allow for larger space and time steps while retaining accuracy and efficiency across scales \cite{Russo,JinAP1,JinAP2,JinAP3,JinPareschiToscani1998,JinPareschiToscani1999}.

The objective of this work is to develop a general methodology for constructing high-order schemes for the incompressible Navier-Stokes equations, derived from suitable discretizations of Lattice-Boltzmann type relaxation systems. Building on earlier ideas \cite{Klar1998, Mapundi,JinPareschiToscani1998,HuangXingXiong2025,Zeifang2020,BoscheriDimarcoLoubere2020}, we propose an arbitrarily high-order approach that remains uniformly stable with respect to the Knudsen and Mach numbers and which, in the low Mach limit, reduces to an IMEX high-order finite difference solver for the incompressible Navier-Stokes equations. The method combines nonoscillatory spatial discretizations with implicit- explicit (IMEX) Runge-Kutta time integrators. In this setting, the nonoscillatory upwind discretization of the convective part of the relaxation system naturally reduces, in the asymptotic limit, to a high-order treatment of the nonlinear convective terms of the incompressible equations. More specifically, we present a family of numerical schemes based on a Lattice-Boltzmann discrete six-velocity system with diffusive scaling, capable of achieving uniform accuracy in the low Mach number limit without refining the discretization in space and time. In the asymptotic limit, the proposed methods become high-order projection methods for the incompressible Navier-Stokes equations, solved on a collocated grid. %The reformulation proposed in this work of the six-velocity system at the same time clarifies the relationship between Lattice-Boltzmann approaches and relaxation-based schemes. 

The rest of the paper is organized as follows. Section \ref{sec2} recalls the connection between
kinetic theory and the incompressible Navier-Stokes equations while Section \ref{sec2.1} introduces the Lattice-Boltzmann type discrete six velocity model with diffusive scaling and the associated closed moment system which constitutes a one to one map between the Lattice Boltzmann and the macrosocopic fluid equations. Section \ref{sec3} is dedicated to the time and space discretizations that yield the proposed family of numerical schemes, based on IMEX Runge-Kutta methods, which in turn, we show to lead to a projection formulation for the incompressible Navier-Stokes equations. High-order spatial discretizations are then applied consistently within this framework. Section \ref{sec4} presents a set of numerical experiments that illustrate the accuracy, efficiency, and stability of the proposed approach. Finally, Section \ref{sec5} draws some conclusions and outlines directions for future research.

\section{Kinetic equations and the incompressible Navier-Stokes limit}\label{sec2}

We introduce here the underlying kinetic model and recall some results concerning its incompressible limit. These results will serve as the starting point for the derivation of the lattice Boltzmann model in the next section and will subsequently be used to establish the numerical connection between the two models. For a theoretical analysis of kinetic equations and their asymptotic limits, we refer to \cite{Levermore91}, while for related numerical discussions we refer to \cite{Klar99,JunkKlar,JinPareschiToscani1998}. We consider the following equation
\begin{equation}
	\partial_t F + v \cdot \nabla_x F = \frac{1}{\varepsilon} C(F), \label{eq:kinetic}
\end{equation}
where $F = F(x,v,t)$ is the so-called distribution function, which gives the probability for a particle to be located at the spatial position $x \in D \subset \mathbb{R}^d, d=1,2,3$, with velocity $v=(v_1,\dots,v_d) \in \mathbb{R}^d$, at time $t \in [0,\infty)$.
The operator $C(F)$ denotes a given collision operator, describing the interactions between particles. The invariants of the collision operator are assumed to be $1, v, |v|^2$, as in the leading example of the Boltzmann collision operator in rarefied gas dynamics.
Another well-known example is given by the BGK model (see \cite{Cer}), which simplifies the interactions by modeling them as a relaxation towards the equilibrium Maxwellian distribution.
Introducing the diffusive space–time scaling $x \mapsto x/\varepsilon$ and $t \mapsto t/\varepsilon^2$, where $\varepsilon$ is the Knudsen number (related to the mean free path between two consecutive collisions), one obtains the scaled equation
\begin{equation}
	\partial_t F + \frac{1}{\varepsilon} v \cdot \nabla_x F
	= \frac{1}{\varepsilon^2} C(F). \label{eq:scaled}
\end{equation}
We also introduce the Mach number $Ma$, denoting the ratio of the bulk velocity to the sound speed, and the Reynolds number $Re$, a dimensionless parameter representing the ratio of inertial to viscous forces within a fluid.
These numbers are related to $\varepsilon$ by the following relation:
\begin{equation}
	\varepsilon = \frac{Ma}{Re}. \label{eq:eps-ma-re}
\end{equation}
Thanks to a perturbation procedure (see, for example, \cite{Levermore91,Esposito}),
the limit equation for \eqref{eq:scaled} as $\varepsilon \to 0$ is obtained by searching for solutions of the form
\begin{equation}\label{eq:ansatz}
	F = M \big(1 + \varepsilon f \big),
\end{equation}
where $M$ is the normalized Maxwellian with zero drift,
\[
M(v) = \frac{1}{(2\pi)^{d/2}} \exp\!\left(-\frac{|v|^2}{2}\right).
\]
Substituting this ansatz into \eqref{eq:scaled} yields
\begin{equation}
	\varepsilon^2 \partial_t f + \varepsilon v \cdot \nabla_x f
	= L f + \varepsilon Q(f,f) + \mathcal{O}(\varepsilon^2), \label{eq:expansion}
\end{equation}
where $L$ denotes the linearized collision operator,
\[
Lf = M^{-1} D C(M) M f,
\]
and $Q$ is given by the second Fréchet derivative of $C$ around $M$:
\[
Q(f,f) = \frac{1}{2} M^{-1} D^2 C(M):(Mf \vee Mf),
\]
with $:$ denoting the contracted tensor product and $\vee$ the symmetric tensor product. In the following, we focus on the specific case of the BGK collision operator \cite{BGK}, which reads
\begin{equation}
	C(F)(v) = -\tfrac{1}{\tau} \Big( F(v) - M_F(v) \Big), \label{eq:bgk}
\end{equation}
where $M_F$ is the Maxwellian having the same moments with respect to $1, v, |v|^2$ as $F$, and $\tau$ is the collision frequency.
In this case one obtains
\[
L(f)(v) = -\frac{1}{\tau} (f - Pf)(v).
\]
where we have defined the projection onto the space of collision invariants,
\begin{equation}
	Pf = \langle f \rangle
	+ v \cdot \langle v f \rangle
	+ \left\langle \frac{|v|^2 - d}{d} f \right\rangle
	\frac{|v|^2 - d}{2}, \label{eq:projection}
\end{equation}
with
\[
\langle f \rangle = \int_{\mathbb{R}^d} f(v) M(v) \, dv
\]
and with $d$ the dimension of the velocity space, as already stated. Writing now $f$ as
\[
f = f_0 + \varepsilon f_1 + \varepsilon^2 f_2 + \cdots,
\]
and inserting into \eqref{eq:expansion}, collecting terms of equal power in $\varepsilon$, 
gives equations for $f_0, f_1, f_2, \dots$.
The solvability conditions for these equations yield
\[
f_0 = \rho_0 + v \cdot u_0 + \frac{|v|^2 - d}{2} T_0,
\]
where $u_0$ solves the incompressible Navier--Stokes equations
\begin{equation}
	\begin{aligned}
		\partial_t u_0 + u_0 \cdot \nabla_x u_0 + \nabla_x p &= \mu \Delta u_0, \\
		\nabla_x \cdot u_0 &= 0,
	\end{aligned} \label{eq:ns}
\end{equation}
with pressure $p$, temperature $T_0$ and density $\rho_0$ determined by a heat transfer equation and 
the Boussinesq relation (see \cite{Levermore91}).  
Finally, the viscosity coefficient $\mu$ is determined by the linearized collision operator $L$ and in the case of the BGK model is such that $\mu = \tau$.

\subsection{Lattice Boltzmann discrete velocity model}\label{sec2.1}
Taking inspiration from the BGK model \eqref{eq:kinetic},\eqref{eq:bgk}, we now derive a relaxation system under a diffusive scaling based on a lattice Boltzmann–type discrete velocity formulation.
In particular, we consider a six-velocity lattice on a hexagonal grid in two spatial dimensions.
As will be shown later, this approach leads to the incompressible Navier–Stokes equations in the low Mach number limit \cite{JUNK,Klar99}. To that aim, we introduce the following finite set of discrete velocities
\begin{align}
	v \in \{ c_1,\dots, c_N\}, \quad c_i \in \mathbb{R}^2.
\end{align}
Then, by taking $N=6$, we define the resulting six velocities as
\begin{align}
	c_i = \left(c_i^{(1)},c_i^{(2)} \right) = \left( \cos(\pi(i-1)/3),\sin(\pi(i-1)/3)\right), \quad i =1,\dots, 6.
\end{align}
At the same time, we define the following $N=6$ functions
\begin{align}
	F_i(x,t) = F(x,c_i,t),\quad i = 1,\dots,N
\end{align}
which scope is to mimic the probability distribution function at space position $x$ and velocity $c_i$ defined in the previous section. The Lattice Boltzmann discrete velocity model is then given by
\begin{align} \label{DiscVelocityModel2}
	\partial_t F_i + c_i \cdot \nabla_x F_i = \frac{1}{\varepsilon} C(F_i),
\end{align}
where the collision operator is of BGK-type and reads
\begin{align}
	C(F_i) = - \frac{1}{\tau} \left( F_i -F_i^{eq}\right),
\end{align}
with $\tau$ a suitable coefficient. In this model, the equilibrium distribution function $F_i^{eq}$ is defined by
\begin{align*}
	F_i^{eq} = \frac{\rho}{6} + \frac{1}{3} c_i \cdot u + \frac{2}{3} \left( c_i \otimes c_i - \frac{1}{2}I\right): u \otimes u,
\end{align*}
and it is constructed in such a way that in the diffusive limit the incompressible Navier-Stokes equations is recovered. The moments leading to the macroscopic equations are obtained in the same spirit of the moments of the distribution function $F$, i.e. by approximating the integrals over the continuous velocity space of the original model \eqref{eq:kinetic} with summations. Thus, in particular, the mass density is obtained by direct summation of the functions $F_i$, i.e.
\begin{align}\label{density}
	\rho = \sum_{i = 1}^N F_i(x,t)
\end{align}
while the momentum is obtained by multiplying each distribution function $F_i$ by the corresponding lattice velocity and then summing over all directions:
\begin{align}\label{momentum}
	u_k = \sum_{i = 1}^N c_i^{(k)} F_i(x,t), \quad k = 1,2.
\end{align}
%Next, we want to rewrite \eqref{DiscVelocityModel2} into an equivalent system of moment equations in order to being able to solve it numerically.\\
%The moments related to the collision invariants $1, c_i^{(1)}, c_i^{(2)}$ are $\rho$ and $u_1,u_2$. Instead of $1$ and the related moment $\rho$, we can also choose $|c_i|^2$ and 
We also define the following second order moments useful in the following
\begin{align*}
	\theta = \sum_{i=1}^N \frac{|c_i|^2}{2}F_i - \frac{\bar{\rho}}{2} = \frac{1}{2} \left(\sum_{i=1}^N F_i - \bar{\rho} \right),
\end{align*}
where $\bar{\rho}$ is an arbitrary constant that will be fixed later.
Now, with the aim of reformulating the Lattice–Boltzmann model in an equivalent form, where the unknowns are expressed in terms of the moments of the distribution functions $F_i$, we introduce three additional second order moments:
\begin{align}
	v_1 &=  \sum_{i=1}^N \left( ( c_i^{(1)})^2 - \frac{|c_i|^2}{2}\right) F_i = \sum_{i=1}^N \left( ( c_i^{(1)})^2 - \frac{1}{2}\right) F_i, \\
	v_2 &= \sum_{i=1}^N c_i^{(1)}c_i^{(2)} F_i, \\
	p &= \theta - \frac{|u|^2}{2}
\end{align}
and finally one third order moment
\begin{align}\label{q}
	q = \sum_{i=1}^N c_i^{(1)} \left(( c_i^{(2)})^2 - \frac{1}{4} \right) F_i.
\end{align}
Let us observe that, since the space orthogonal to the space of collision invariants is spanned by 
\begin{equation}\label{mom}
( c_i^{(1)})^2 - \frac{|c_i|^2}{2}, \quad c_i^{(1)}c_i^{(2)}, \quad c_i^{(1)}\left(( c_i^{(2)})^2 - \frac{1}{4} \right),
\end{equation}
it follows that there exists a one–to–one correspondence between the distribution functions $F_i$, $i=1,\dots,N$, and the moments $u_1, u_2, \theta, v_1, v_2, q$ defined above. Now, in analogy with the continuous model discussed in the previous chapter, we introduce the diffusive space–time scaling
\begin{align} \label{DiscScaledEqu}
	\partial_t F_i + \frac{1}{\varepsilon} c_i \cdot \nabla_x F_i = -\frac{1}{\varepsilon^2}C(F_i).
\end{align} 
We then insert the ansatz $F_i = M_i(1+ \varepsilon f_i)$ into \eqref{DiscScaledEqu}, where the normalised Maxwellian $M_i$ becomes in this case just equal to $\frac{1}{6}$. This leads to
\begin{align}\label{DiscScaledequ2}
	\partial_t f_i + \frac{1}{\varepsilon} c_i \cdot \nabla_x f_i = -\frac{1}{\varepsilon^2}C(f_i),
\end{align}
with
\begin{align}
	C(f_i) = L(f_i) + \varepsilon Q(f_i,f_i),
\end{align}
where
\begin{align}
	L(f_i) = - \frac{1}{\tau} \left(f_i - f_i^{leq} \right),
\end{align}
with the linearised equilibrium distribution
\begin{align}
	f_i^{leq} = \frac{\rho}{6} + \frac{1}{3} c_i \cdot u
\end{align}
and 
\begin{align}
	Q(f_i,f_i) = \frac{2}{3 \tau} \left(c_i \otimes c_i - \frac{1}{2}I \right): u \otimes u.
\end{align}
From the above definitions and the chosen ansatz, we observe that letting $\varepsilon \to 0$ implies $F_i \to F_i^{eq}$. Thus, $v_1, v_2, q \to 0$ as $\varepsilon \to 0$.
In the same limit, one also obtains $\nabla_x \rho \to 0$, which means that, under suitable boundary conditions, the density becomes constant in both space and time. We now proceed by rescaling the quantities
\begin{align*}
	\theta &= \frac{1}{2\varepsilon} \left(\sum_{i=1}^N f_i - \bar{\rho} \right) \\
	v_1 &= \frac{1}{\varepsilon} \sum_{i=1}^N \left( (c_i^{(1)})^2 - \frac{1}{2}\right)f_i \\
	v_2 &= \frac{1}{\varepsilon} \sum_{i=1}^N c_i^{(1)}c_i^{(2)}f_i \\
	q&= \frac{1}{\varepsilon^2} \sum_{i=1}^N c_i^{(1)} \left((c_i^{(2)})^2 -\frac{1}{4} \right)f_i.
\end{align*}
Then, multiplying \eqref{DiscScaledequ2} by $c_i^{(1)},c_i^{(2)}$ and by the quantities of \eqref{mom}, choosing $\bar{\rho}$ to be the constant limit of $\rho$ as $\varepsilon \rightarrow 0$,
one gets the corresponding moment system which reads
\begin{align} \label{ScaledMomentSystem}
	\partial_t u_1 - \partial_x v_1 + \partial_y v_2 + \partial_x \theta &= 0 \\
	\partial_t u_2 + \partial_x v_2 + \partial_y v_1 + \partial_y \theta &= 0 \\
	\partial_t \theta + \frac{1}{2\varepsilon^2} \left( \partial_x u_1 + \partial_y u_2 \right) &= 0 \\
	\partial_t v_1 + \frac{1}{4\varepsilon^2} \left( - \partial_x u_1 + \partial_y u_2 \right) +  \partial_x q &= - \frac{1}{\varepsilon^2\tau} \left( v_1 + \frac{1}{2} \left(u_1^2 - u_2^2 \right)\right) \\
	\partial_t v_2 + \frac{1}{4\varepsilon^2} \left(  \partial_x u_2 + \partial_y u_1 \right) +  \partial_y q &= - \frac{1}{\varepsilon^2\tau} \left( v_2 -  u_1 u_2\right) \\ \label{ScaledMomentSystem1}
	\partial_t q + \frac{1}{2\varepsilon^2} \left( \partial_x v_1 + \partial_x v_2 \right) &= - \frac{1}{\varepsilon^2\tau} q.
\end{align}
The diffusion limit can be now computed straightforwardly. For $\varepsilon \rightarrow 0$, the equation of $\theta$ becomes the incompressibility condition
\begin{align}
	\nabla \cdot u = 0.
\end{align}
The equations for $v_1$ and $v_2$ give
\begin{align}
	v_1 &= -\frac{1}{2}\left( u_1^2 - u_2^2\right) -\frac{\tau}{4} \left(-\partial_x u_1 + \partial_y u_2 \right), \\
	v_2 &= u_1u_2 - \frac{\tau}{4}\left( \partial_x u_2 + \partial_y u_1 \right).
\end{align}
This yields
\begin{align}
	\partial_t u_1 + \partial_x  \frac{1}{2}\left( u_1^2 - u_2^2\right)  +\partial_y \left(u_1u_2  \right) 
	+  \partial_x \theta &= \frac{\tau}{4} \Delta u_1 \\
	\partial_t u_2 + \partial_x  u_1u_2 -\partial_y \frac{1}{2}\left( u_1^2 - u_2^2\right) 
	+  \partial_y \theta &= \frac{\tau}{4} \Delta u_2.
\end{align}
By using the definition of $p$ and vector notation, we obtain the incompressible Navier Stokes equations with $Re = \frac{4}{\tau}$:
\begin{align}\label{INS1}
	\partial_t u + (u \cdot \nabla_x)u + \nabla_x p &= \frac{\tau}{4} \Delta u \\
	\nabla_x \cdot u &= 0. \label{INS2}
\end{align}
The above formal derivation can be rigorously proven using a stability concept developed in \cite{Yong2001}. We refer to \cite{Junk_Yong} and \cite{Banda_Yong_Klar} for further details and convergence proofs.

\section{A class of numerical methods for the Lattice-Boltzmann model in the diffusion scaling}\label{sec3}
The purpose of this section is to introduce a family of numerical methods which, starting from the lattice Boltzmann model defined in the previous section, provide uniformly stable approximations with respect to the small parameter $\varepsilon$. In the asymptotic limit 
$\varepsilon\to 0$, these methods yield high–order accurate approximations of the incompressible Navier–Stokes equations. We start by rewriting the closed moment system \eqref{ScaledMomentSystem}-\eqref{ScaledMomentSystem1} derived in the previous section in vector form. This reads
\begin{equation}\begin{split} \label{VectorForm}
		\partial_t u + \nabla_x \cdot B(v) + \nabla_x \theta &= 0 \\
		\partial_t \theta + \frac{1}{2 \varepsilon^2} \nabla_x \cdot u &= 0 \\
		\partial_t v + \frac{1}{4\varepsilon^2} \nabla_x \cdot B(u) + \nabla_x q &= - \frac{1}{\varepsilon^2 \tau} \left(v- F(u) \right)\\
		\partial_t q + \frac{1}{2\varepsilon^2} \nabla_x \cdot v &= - \frac{1}{\varepsilon^2 \tau} q
	\end{split}
\end{equation}
with 
\[v = \begin{pmatrix}
	v_1 \\ v_2
\end{pmatrix},\quad u = \begin{pmatrix}
	u_1 \\ u_2
\end{pmatrix}, \quad F(u) = \begin{pmatrix}
	\frac{1}{2}(u_2^2-u_1^2) \\ u_1u_2
\end{pmatrix}, \quad B(\omega) = \begin{pmatrix}
-\omega_1 & \omega_2 \\ \omega_2 & \omega_1
\end{pmatrix} 
\]
where $\omega$ represents any two dimensional vector. In the following section, we first introduce the time discretization and then we discuss the properties of the semi-discrete system while in section \ref{sec:sd} we apply a proper space discretization to the resulting semi-discrete scheme.

\subsection{Time discretization}
The main idea is to use a suitable combination of implicit–explicit (IMEX) Runge–Kutta discretizations \cite{IMEX_PR,Bosc1}. The main advantage of the proposed strategy is that, compared to a fully explicit method, the time step is not restricted by the stiffness of the system, while compared to a fully implicit scheme, the solution procedure is more straightforward, since the system to be solved is less involved. We will show that this choice is sufficient to obtain, in the mean–free–path limit, a projection scheme for the incompressible Navier–Stokes equations. The general IMEX–RK schemes can be conveniently represented by a double Butcher tableau of the following form:
\begin{align}
	\text{Explicit: } 
	\begin{array}{c|c}
		\tilde{c} & \tilde{A} \\ \hline \\ [-1.9ex]
		& \tilde{w}^T
	\end{array}
	\qquad
	\text{Implicit: }
	\begin{array}{c|c}
		c & A \\ \hline \\[-1.9ex]
		& w^T
	\end{array}
\end{align}
For the explicit scheme, $\tilde{A}$ is an $s \times s$ lower–triangular matrix with zero diagonal entries.
For the implicit scheme, the $s \times s$ matrix $A$ is restricted to the class of diagonally implicit Runge–Kutta (DIRK) methods, which means that $a_{ij} = 0$ for $j > i$. The vectors $w$ and $\tilde{w}$ contain the quadrature weights used to combine the internal stages of the Runge–Kutta method, while the coefficients $c$ and $\tilde{c}$ are defined as
\[c_i = \sum_{j=1}^{i} a_{ij}, \quad \tilde{c}_i = \sum_{j=1}^{i-1} \tilde{a}_{ij}.\]
The following definitions characterize the IMEX schemes and the properties we aim to consider in this work. For further details we refer for instance to \cite{IMEX_PR}.
\begin{definition}
	An IMEX RK method is said to be of type A, if the matrix $A \in \mathbb{R}^{s \times s}$ is invertible, or equivalently $a_{ii} \neq 0$, $i = 1,\dots,s$.
\end{definition}
\begin{definition}
	An IMEX RK method is said to be of type CK, if the matrix $A \in \mathbb{R}^{s \times s}$ can be written as
	\begin{align*}A=\left(
		\begin{matrix}
			0 & 0 \\ a & \hat{A}
		\end{matrix}\right)
	\end{align*}
	with $a = (a_{21},\dots,a_{s1})^T \in \mathbb{R}^{(s-1)}$ and the submatrix $\hat{A} \in \mathbb{R}^{(s-1) \times (s-1)}$ invertible. %If additionally $a = 0$, the scheme is said to be of type ARS.
\end{definition}
\begin{definition}
	An IMEX RK method is called implicitly stiffly accurate (ISA), if the corresponding DIRK method is stiffly accurate, which means that
	\begin{align*}
		a_{si} = w_i, \qquad i = 1,\dots,s.
	\end{align*}
	If additionally the explicit method satisfies
	\begin{align}
		\tilde{a}_{si} = \tilde{w}_i, \qquad i = 1,\dots,s-1,
	\end{align}
	the IMEX RK method is called globally stiffly accurate (GSA).
\end{definition}
Note that the numerical solution of a GSA IMEX RK scheme coincides exactly with the last internal stage \cite{IMEX_PR},\cite{Hairer}. We now apply the above general IMEX-RK approach to the moment system \eqref{VectorForm}. We proceed by separating the implicit and the explicit components as following
\begin{equation}
\begin{split} \label{implicit-explicit}
		\partial_t u +  \underbrace{\nabla_x \cdot B(v) + \nabla_x \theta}_{implicit} &= 0 \\
		\partial_t \theta + \underbrace{\frac{1}{2 \varepsilon^2} \nabla_x \cdot u}_{implicit} &= 0 \\
		\partial_t v +  \underbrace{\frac{1}{4\varepsilon^2} \nabla_x \cdot B(u) + \nabla_x q}_{explicit} &= - \frac{1}{\varepsilon^2 \tau} \left( \underbrace{v}_{implicit}-  \underbrace{F(u)}_{explicit} \right)\\
		\partial_t q +  \underbrace{\frac{1}{2\varepsilon^2} \nabla_x \cdot v}_{implicit} &= -  \underbrace{\frac{1}{\varepsilon^2 \tau} q}_{implicit}
	\end{split}
\end{equation}
This reads for the internal stages
\begin{equation}
	\begin{split} \label{InternalStages}
		u^{(i)} &= u^n - \Delta t \left[\sum_{j=1}^i a_{ij} \nabla_x \cdot B(v^{(j)}) + \sum_{j=1}^i a_{ij} \nabla_x  \theta^{(j)} \right] \\
		\theta^{(i)} &= \theta^n - \Delta t\left[\frac{1}{2 \varepsilon^2}\sum_{j=1}^{i-1} a_{ij} \nabla_x \cdot u^{(j)} \right] - \Delta t \frac{1}{2\varepsilon^2} a_{ii} \nabla_x \cdot u^{(i)} \\
		v^{(i)} &= v^n - \Delta t \left[ \sum_{j=1}^{i-1} \tilde{a}_{ij} \left( \frac{1}{4 \varepsilon^2} \nabla_x \cdot B(u^{(j)}) + \nabla_x q^{(j)} - \frac{1}{\tau \varepsilon^2} F(u^{(j)})\right)\right]\\ &- \Delta t \sum_{j=1}^{i-1} a_{ij} \frac{1}{\tau \varepsilon^2} v^{(j)} -  \Delta t  a_{ii} \frac{1}{\tau \varepsilon^2} v^{(i)} \\
		q^{(i)} &= q^n - \Delta t \left[ \sum_{j=1}^i a_{ij} \frac{1}{2\varepsilon^2} \nabla_x \cdot v^{(j)} + \sum_{j=1}^{i-1} a_{ij} \frac{1}{\tau \varepsilon^2} q^{(j)}\right] - \Delta t \frac{a_{ii}}{\tau \varepsilon^2} q^{(i)}
	\end{split}
\end{equation}
while the numerical solution reads
\begin{equation}\label{FinalStep}
	\begin{split}
		u^{n+1} &= u^n - \Delta t \left[ \sum_{i=1}^s w_i \left( \nabla_x \cdot B(v^{(i)}) + \nabla_x \theta^{(i)} \right)\right] \\
		\theta^{n+1} &= \theta^n  - \Delta t \left[ \sum_{i=1}^s w_i \frac{1}{2\varepsilon^2} \nabla_x u^{(i)}\right] \\
		v^{n+1} &= v^n  - \Delta t \left[ \sum_{i=1}^{s} \tilde{w}_{i} \left( \frac{1}{4 \varepsilon^2} \nabla_x \cdot B(u^{(i)}) + \nabla_x q^{(i)} - \frac{1}{\tau \varepsilon^2} F(u^{(i)})\right)+ \sum_{i=1}^{s} w_{i} \frac{1}{\tau \varepsilon^2} v^{(i)}\right]\\
		q^{n+1} &= q^n  - \Delta t \left[ \sum_{i=1}^s w_{i} \left( \frac{1}{2\varepsilon^2} \nabla_x \cdot v^{(i)} +  \frac{1}{\tau \varepsilon^2} q^{(i)}\right)\right].
	\end{split}
\end{equation}
Let observe that since we are only considering GSA IMEX RK schemes, it is sufficient to compute the internal stages, because the last internal stages coincide with the numerical solution. We are now able to prove the following general result for the semi-discrete schemes.
\begin{theorem}
	If the IMEX method is of type A and satisfies the GSA property, then in the fluid dynamic limit $\varepsilon \rightarrow 0$ the IMEX RK scheme \eqref{InternalStages} becomes a consistent discretization for the incompressible Navier Stokes equations. %if the initial condition is such that for $\varepsilon\to 0$, one gets $\nabla\cdot u(t=0)=0$.
\end{theorem}

\begin{proof}
	By rewriting the equation of the internal stages $v^{(i)}$, we obtain
	\begin{align}
		v^{(i)} &= \frac{\tau \varepsilon^2}{\tau \varepsilon^2 + \Delta t a_{ii}}v^n \nonumber \\&- \Delta t \frac{1}{\tau \varepsilon^2 + \Delta t a_{ii}}\left[ \sum_{j=1}^{i-1} \tilde{a}_{ij} \left( \frac{\tau }{4} \nabla_x \cdot B(u^{(j)}) + \tau \varepsilon^2 \nabla_x q^{(j)} -  F(u^{(j)})\right)\right] \label{V_i}\\ \nonumber &- \Delta t \sum_{j=1}^{i-1} a_{ij} \frac{1}{\tau \varepsilon^2 + \Delta t a_{ii}} v^{(j)}.
	\end{align}
%	Next, we insert \eqref{V_i} into the last internal stage of $u$.
%	\begin{align}
%		u^{(s)} &= u^n - \Delta t \left[\sum_{j=1}^{s-1} a_{sj} \nabla_x \cdot B(v^{(j)}) + \sum_{j=1}^s a_{sj} \nabla_x  \theta^{(j)} \right] - \Delta t a_{ss}\nabla_x \cdot B(v^{(s)}) 
%	\end{align}
Then, by inserting the above expression into the first equation of \eqref{InternalStages}
and by taking the limit $\varepsilon \rightarrow 0$, we obtain % by inserting the last stage of \eqref{V_i}
	\begin{align}\nonumber
		u^{(i)} &= u^n - \Delta t \left[\sum_{j=1}^{i-1} a_{ij} \nabla_x \cdot B(v^{(j)}) + \sum_{j=1}^i a_{ij} \nabla_x  \theta^{(j)} \right] \\ &- \Delta t \nabla_x \cdot B\left( -\sum_{j=1}^{i-1} \tilde{a}_{ij} \left( \frac{\tau }{4} \nabla_x \cdot B(u^{(j)}) -  F(u^{(j)})\right) - \sum_{j=1}^{i-1} a_{ij} v^{(j)} \right).\label{u_i}
	\end{align}
	Rewriting gives
	\begin{align}\nonumber
		u^{(i)} &= u^n - \Delta t \left[\sum_{j=1}^{i-1} a_{ij} \nabla_x \cdot B(v^{(j)}) + \sum_{j=1}^i a_{ij} \nabla_x  \theta^{(j)} \right] \\ &- \Delta t \nabla_x \cdot B\left( -\sum_{j=1}^{i-1} \tilde{a}_{ij} \left( \frac{\tau }{4} \nabla_x \cdot B(u^{(j)}) -  F(u^{(j)})\right) \right)  + \Delta t \sum_{j=1}^{i-1} a_{ij} \nabla_x \cdot B (v^{(j)}) .\label{u_eq}
	\end{align}
We now consider the stage equations for $\theta$. Substituting $u^{(i)}$ using \eqref{InternalStages}, gives 
\begin{align}
	\theta^{(i)} = \theta^n - \Delta t\left[\frac{1}{2 \varepsilon^2}\sum_{j=1}^{i-1} a_{ij} \nabla_x \cdot u^{(j)} \right] - \Delta t \frac{1}{2\varepsilon^2} a_{ii} \nabla_x \cdot u^{(i)}_* + \Delta t^2 \frac{1}{2\varepsilon^2} a_{ii}^2 \Delta_x  \theta^{(i)},
\end{align}
where
\begin{align}
	u^{(i)}_* &= u^n - \Delta t \left[\sum_{j=1}^i a_{ij} \nabla_x \cdot B(v^{(j)}) + \sum_{j=1}^{i-1} a_{ij} \nabla_x  \theta^{(j)} \right].
\end{align}
Rewriting gives
\begin{align}
	\Delta_x \theta^{(i)} - \frac{2\varepsilon^2}{\Delta t^2 a_{ii}^2} \theta^{(i)} = - \frac{2\varepsilon^2}{\Delta t^2 a_{ii}^2} \theta^n + \frac{1}{\Delta t a_{ii}^2} \sum_{j=1}^{i-1} a_{ij} \nabla_x \cdot u^{(j)} +  \frac{1}{\Delta t a_{ii}} \nabla_x \cdot u^{(i)}_*. \label{Theta_i}
\end{align}
%% Hier OK!!
Now, in the limit $\varepsilon \rightarrow 0$, we get for the first stage 
\begin{align}
\Delta_x \theta^{(1)}=\frac{1}{\Delta t a_{11}} \nabla_x \cdot u^{n}- \nabla_x \cdot \nabla_x \cdot (B(v^{(1)})).	
\end{align}
 Taking the divergence of \eqref{u_eq} in the case of the first stage gives
	 \begin{align}
		\nabla_x \cdot u^{(1)} = \nabla_x \cdot u^n - \nabla_x \cdot u^{n}+ \Delta t a_{11} \nabla_x \cdot \nabla_x \cdot (B(v^{(1)}))=0. \label{induction_start}
	\end{align}
For the inductive step $(i-1) \mapsto i$, \eqref{u_i} gives:
	 \begin{align}\nonumber
		\nabla_x \cdot u^{(i)} &= \nabla_x \cdot u^n - \Delta t \left[\sum_{j=1}^{i-1} a_{ij} \nabla_x \cdot \nabla_x \cdot B(v^{(j)}) + \sum_{j=1}^{i} a_{ij} \Delta_x  \theta^{(j)} \right] \\ &+ \Delta t \nabla_x \cdot \nabla_x \cdot B\left( \sum_{j=1}^{i-1} \tilde{a}_{ij} \left( \frac{\tau }{4} \nabla_x \cdot B(u^{(j)}) -  F(u^{(j)})\right) \right)  + \Delta t \sum_{j=1}^{i-1} a_{ij}\nabla_x \cdot \nabla_x \cdot B (v^{(j)}) . \label{induction_step}
	\end{align}	
Now, using \eqref{Theta_i} and the induction hypothesis $\nabla_x \cdot u^{(i-1)} = 0$ gives in the limit 
	\begin{align}
	\sum_{j=1}^{i} a_{ij} \Delta_x  \theta^{(j)} = \sum_{j=1}^{i-1} a_{ij} \Delta_x  \theta^{(j)} + \frac{1}{\Delta t} \left( \nabla_x \cdot u^n - \Delta t \left(\sum_{j=1}^{i}a_{ij}\nabla_x \cdot \nabla_x \cdot B(v^{(j)}) +\sum_{j=1}^{i-1} a_{ij} \Delta_x  \theta^{(j)} \right)\right). \label{sum_lapl_theta}
	\end{align}
	While, for $\varepsilon \rightarrow 0$, equation \eqref{V_i} gives
	\begin{align}
	v^{(j)} = - \frac{1}{a_{jj}} \left( \sum_{k=1}^{j-1} \tilde{a}_{jk} \left(\frac{\tau}{4} \nabla_x \cdot B(u^{(k)}) - F(u^{(k)}) \right)  + \sum_{k=1}^{j-1} a_{jk} v^{(k)}\right). \label{v_j_eps_0}
	\end{align}
	Plugging \eqref{sum_lapl_theta} and \eqref{v_j_eps_0} into \eqref{induction_step}, we obtain in the incompressible limit
	\begin{align}
	\nabla_x \cdot u^{(i)} = 0.
	\end{align}
	Now, by using the following properties
	\begin{align}
		\nabla_x \cdot B(\nabla_x \cdot B(u)) &= \Delta_x u \\
		\nabla_x \cdot B(F(u)) &= (u\cdot \nabla_x) u - \nabla_x \left( \frac{|u|^2}{2}\right) + u \nabla \cdot u
	\end{align}
	and the divergence–free condition for all internal stages of $u$, we get
	\begin{align}\label{u_s}
		u^{(i)} &= u^n - \Delta t \left[ \sum_{j=1}^i a_{ij} \nabla_x  \theta^{(j)} \right] \nonumber \\ &- \Delta t  \sum_{j=1}^{i-1} \tilde{a}_{ij} \left( -\frac{\tau }{4} \Delta_x u^{(j)} +  (u^{(j)}\cdot \nabla_x) u^{(j)} - \nabla_x \left( \frac{|u^{(j)}|^2}{2}\right)\right).  
	\end{align}
Due to the fact that we use an GSA IMEX RK scheme, the above relation is enough to state that the numerical solution becomes 
	\begin{align}\nonumber
		u^{n+1} &= u^n - \Delta t \left[ \sum_{i=1}^s w_{i} \nabla_x  \theta^{(i)} \right] \\ &- \Delta t  \sum_{i=1}^{s} \tilde{w}_{i} \left( -\frac{\tau }{4} \Delta_x u^{(i)} +  (u^{(i)}\cdot \nabla_x) u^{(i)} - \nabla_x \left( \frac{|u^{(i)}|^2}{2}\right)\right).  
	\end{align}
Finally, because by definition, we have
	\begin{align}\label{Theta_P_Prop}
		\theta^{(i)} = p^{(i)} + \frac{|u^{(i)}|^2}{2}
	\end{align}
and since $\sum_{i=1}^s \tilde{w}_i = \sum_{i=1}^s w_i$, we obtain for the numerical solution of the velocity field $u$ the following expression
	\begin{align}\label{Glg1}
		u^{n+1} &= u^n - \Delta t \left( \sum_{i=1}^{s} \tilde{w}_{i} \left( -\frac{\tau }{4} \Delta_x u^{(i)} +  (u^{(i)}\cdot \nabla_x) u^{(i)} \right)
		+  \sum_{i=1}^s w_{i} \nabla_x  p^{(i)} \right), 
	\end{align}
which is a consistent high order Implicit-Explicit time discretization of equation \eqref{INS1}.
We now consider the rewritten stage equations for $\theta$ \eqref{Theta_i}.
%	\begin{align}
%		\theta^{(i)} = \theta^n - \Delta t\left[\frac{1}{2 \varepsilon^2}\sum_{j=1}^{i-1} a_{ij} \nabla_x \cdot u^{(j)} \right] - \Delta t \frac{1}{2\varepsilon^2} a_{ii} \nabla_x \cdot u^{(i)}_* + \Delta t^2 \frac{1}{2\varepsilon^2} a_{ii}^2 \Delta_x  \theta^{(i)},
%	\end{align}
%	where
%	\begin{align}
%		u^{(i)}_* &= u^n - \Delta t \left[\sum_{j=1}^i a_{ij} \nabla_x \cdot B(v^{(j)}) + \sum_{j=1}^{i-1} a_{ij} \nabla_x  \theta^{(j)} \right].
%	\end{align}
%	Rewriting gives
%	\begin{align}
%		\Delta_x \theta^{(i)} - \frac{2\varepsilon^2}{\Delta t^2 a_{ii}^2} \theta^{(i)} = - \frac{2\varepsilon^2}{\Delta t^2 a_{ii}^2} \theta^n + \frac{1}{\Delta t a_{ii}^2} \sum_{j=1}^{i-1} a_{ij} \nabla_x \cdot u^{(j)} +  \frac{1}{\Delta t a_{ii}} \nabla_x \cdot u^{(i)}_*. \label{Theta_i2}
%	\end{align}
Taking the limit $\varepsilon \rightarrow 0$ and by using \eqref{u_s} as well as the divergence-free condition for all $u^{(i)}$, gives
	\begin{align}%\nonumber
		\Delta_x \theta^{(i)}  &=   \frac{1}{ a_{ii}} \nabla_x \cdot \left[ \sum_{j=1}^{i-1} \tilde{a}_{ij} \left( \frac{\tau }{4} \Delta_x u^{(j)} -  (u^{(j)}\cdot \nabla_x) u^{(j)} + \nabla_x \left( \frac{|u^{(j)}|^2}{2}\right)\right) \right]
		- \sum_{j=1}^{i-1} a_{ij} \Delta_x  \theta^{(j)}.
	\end{align}
The above equation can be recast as
	\begin{align}
		a_{ii}\Delta_x \theta^{(i)}  &= \nabla_x \cdot \left[ \sum_{j=1}^{i-1} \tilde{a}_{ij} \left( \frac{\tau }{4} \Delta_x u^{(j)} -  (u^{(j)}\cdot \nabla_x) u^{(j)}\right) \right] +\sum_{j=1}^{i-1} \tilde{a}_{ij} \Delta_x \frac{|u^{(j)}|^2}{2}
		- \sum_{j=1}^{i-1} a_{ij} \Delta_x  \theta^{(j)}
	\end{align}
Now, due to the fact that a GSA IMEX RK scheme is employed, this is equivalent to
	\begin{align}
		\sum_{i=1}^{s} w_{i} \Delta_x  \theta^{(i)} = \nabla_x \cdot \left[ \sum_{i=1}^{s} \tilde{w}_{i} \left( \frac{\tau }{4} \Delta_x u^{(i)} -  (u^{(i)}\cdot \nabla_x) u^{(i)}\right) \right] +\sum_{i=1}^{s} \tilde{w}_{i} \Delta_x \frac{|u^{(i)}|^2}{2}.
	\end{align}
Thanks to \eqref{Theta_P_Prop} and $\sum_{i=1}^s w_i = \sum_{i=1}^s \tilde{w}_i$, we finally get
	\begin{align}\label{Glg2}
		\sum_{i=1}^{s} w_{i} \Delta_x  p^{(i)} = \nabla_x \cdot \left[ \sum_{i=1}^{s} \tilde{w}_{i} \left( \frac{\tau }{4} \Delta_x u^{(i)} -  (u^{(i)}\cdot \nabla_x) u^{(i)}\right) \right],
	\end{align}
which is a consistent high order Implicit-Explicit time discretization of equation \eqref{INS2}.
\end{proof}
We can now state a similar result for IMEX methods of type $CK$ or $ARS$.
\begin{theorem}
	If the IMEX method is of type CK or ARS, satisfies the GSA property and the initial velocity field is divergence free, then in the limit $\varepsilon \rightarrow 0$ the IMEX RK scheme \eqref{InternalStages} becomes a consistent discretization for the incompressible Navier Stokes equations.
\end{theorem}
The proof follows the same line of reasoning as in Theorem 3.4 and is therefore postponed to the Appendix.
Here, we simply highlight the main difference with respect to IMEX schemes of type A, which do not require the initial condition to be well prepared with respect to the limit model in order to be effective.

In practice the algorithm for solving the lattice-Boltzmann model which employ the above discussed class of GSA IMEX RK schemes \eqref{InternalStages} reads:
\begin{algorithm} [H]\caption{An Asymptotic Preserving IMEX Runge-Kutta schemes for the low Mach limit}
%	If the GSA IMEX RK scheme is of type A, compute the internal steps for $i=1,\dots,s$:
	\begin{enumerate}
		\item Compute $v^{(i)}$ by solving
		\begin{align}\nonumber \label{step1}
			v^{(i)} &= \frac{\tau \varepsilon^2}{\tau \varepsilon^2 + \Delta t a_{ii}}v^n \\&- \Delta t \frac{1}{\tau \varepsilon^2 + \Delta t a_{ii}}\left[ \sum_{j=1}^{i-1} \tilde{a}_{ij} \left( \frac{\tau }{4} \nabla_x \cdot B(u^{(j)}) + \tau \varepsilon^2 \nabla_x q^{(j)} -  F(u^{(j)})\right)\right]\\ \nonumber &- \Delta t \sum_{j=1}^{i-1} a_{ij} \frac{1}{\tau \varepsilon^2 + \Delta t a_{ii}} v^{(j)}.
		\end{align}
		\item Compute $\theta^{(i)}$ by solving the following Helmholtz equation
		\begin{align}\nonumber \label{step2}
			\Delta_x \theta^{(i)} - \frac{2 \varepsilon^2}{\Delta t^2 a_{ii}^2} \theta^{(i)} &= - \frac{2 \varepsilon^2}{\Delta t^2 a_{ii}^2} \theta^n + \frac{1}{\Delta t a_{ii}^2}\left[\sum_{j=1}^{i-1} a_{ij} \nabla_x \cdot u^{(j)} \right] \\&+ \frac{1}{\Delta t a_{ii}} \nabla_x \cdot u^{n} - \frac{1}{ a_{ii}} \sum_{j=1}^{i} a_{ij} \nabla_x \cdot \nabla_x \cdot B(v^{(j)}) + \frac{1}{ a_{ii}} \sum_{j=1}^{i-1} a_{ij} \Delta_x  \theta^{(j)}.
		\end{align}
		\item Compute $u^{(i)}$ by solving
		\begin{align}\label{step3}
			u^{(i)} &= u^n - \Delta t \left[\sum_{j=1}^i a_{ij} \nabla_x \cdot B(v^{(j)}) + \sum_{j=1}^i a_{ij} \nabla_x  \theta^{(j)} \right].
		\end{align}
		\item Compute $q^{(i)}$ by solving
		\begin{align} \label{step4}
			q^{(i)} &= \frac{\tau \varepsilon^2}{\tau \varepsilon^2 + \Delta t a_{ii}}q^n - \Delta t \frac{1}{\tau \varepsilon^2 + \Delta t a_{ii}}\left[ \sum_{j=1}^i a_{ij} \frac{\tau}{2} \nabla_x \cdot v^{(j)} + \sum_{j=1}^{i-1} a_{ij}  q^{(j)}\right] 
		\end{align}
	\end{enumerate}
%	else if the GSA IMEX RK scheme is of type CK or ARS, set $v^{(1)} = v^n$, $\theta^{(1)} = \theta^n$, $u^{(1)} = u^n$, $q^{(1)} = q^n$ and for $i=2,\dots,s$ derive the internal stages by performing steps 1-4.
\label{alg1}
\end{algorithm}
%In this paper, we will consider the following GSA IMEX RK schemes:
%\begin{itemize}
%	\item First order GSA Euler IMEX scheme
%	\begin{align}
%		\text{Explicit:  } \begin{array}{c|c}
%			\begin{array}{c}
%				0 \\  1
%			\end{array} &  
%			\begin{array}{cc}
%				0 & 0  \\ 1 & 0  
%			\end{array} \\ \hline   & \begin{array}{cc}
%				1 &  0
%			\end{array}
%		\end{array} \qquad \text{Implicit:  } \begin{array}{c|c}
%			\begin{array}{c}
%				0  \\ 1
%			\end{array} &  
%			\begin{array}{cc}
%				0 & 0 \\ 0 & 1  
%			\end{array} \\ \hline   & \begin{array}{cc}
%				0 & 1 
%			\end{array}
%		\end{array} 
%	\end{align}
%	
%	\item Second order GSA IMEX RK scheme
%	\begin{align}
%		\text{Explicit:  } \begin{array}{c|c}
%			\begin{array}{c}
%				0 \\ c \\ 1
%			\end{array} &  
%			\begin{array}{ccc}
%				0 & 0 & 0 \\ c & 0 & 0 \\ 1-1/(2c) & 1/(2c) & 0
%			\end{array} \\ \hline   & \begin{array}{ccc}
%				1-1/(2c) & 1/(2c) & 0
%			\end{array}
%		\end{array} \qquad \text{Implicit:  } \begin{array}{c|c}
%			\begin{array}{c}
%				0 \\ c \\ 1
%			\end{array} &  
%			\begin{array}{ccc}
%				0 & 0 & 0 \\ 0 & c  & 0 \\ 0 & 1-\gamma & \gamma
%			\end{array} \\ \hline   & \begin{array}{ccc}
%				0 & 1-\gamma & \gamma
%			\end{array}
%		\end{array} 
%	\end{align}
%	with $\gamma = (c-1/2)/(c-1)$ and $c = 2.25$, cf. \cite{boscarino2017}.
%\end{itemize}

%\include{TimeDisc2.tex}

\subsection{Space Discretization} \label{sec:sd}
We introduce a mesh in the two-dimensional physical space with spacing $\Delta x$ and $\Delta y$, and define all variables at the cell centers of the resulting grid. To simplify the presentation and clearly illustrate the basic principle of our spatial discretization, we first describe a first-order finite difference scheme based on Algorithm~\ref{alg1}. The approach for the space discretization follows the methodology proposed in \cite{BOSCARINO}.
\begin{itemize}
	\item The spatial discretization of the first order derivatives $\nabla_x q$, $\nabla_x \cdot B(v)$ in \eqref{step1} and \eqref{step3} of Algorithm~\ref{alg1} are treated by local Lax Friedrichs flux functions. Thus, for $\nabla_x q$, we get 
	\begin{align}
		D_{LF}\ q = \left(\frac{1}{\Delta x} (\hat{q}_{i+1/2,j} - \hat{q}_{i-1/2,j}) , \frac{1}{\Delta y} (\hat{q}_{i,j+1/2} - \hat{q}_{i,j-1/2})\right)
	\end{align}
	with
	\begin{align}\nonumber
		\hat{q}_{i+1/2,j} = \frac{1}{2} \left( q_{i,j} + q_{i+1,j} - \alpha (v_{i+1,j} - v_{i,j})\right), 
		\hat{q}_{i,j+1/2} = \frac{1}{2} \left( q_{i,j} + q_{i,j+1} - \alpha (v_{i,j+1} - v_{i,j})\right)
	\end{align}
	with $\alpha = 1$. The same approach defines the space discretization for $\nabla_x \cdot B(v)$.
	\item The spatial discretization of $\nabla_x \cdot B(u)$ and $\nabla_x \cdot v$ in \eqref{step1} and \eqref{step4} of Algorithm 1 are treated by local Lax Friedrichs flux functions without numerical diffusion:
	\begin{align}
		D_{LF}^0 \cdot v = \frac{1}{\Delta x} (\hat{v}_{i+1/2,j} - \hat{v}_{i-1/2,j}) + \frac{1}{\Delta y} (\hat{v}_{i,j+1/2} - \hat{v}_{i,j-1/2})
	\end{align}
	with
	\begin{align}\nonumber
		\hat{v}_{i+1/2,j} = \frac{1}{2} \left( v_{i,j} + v_{i+1,j} - \alpha_0 (q_{i+1,j} - q_{i,j})\right), 
		\hat{v}_{i,j+1/2} = \frac{1}{2} \left( v_{i,j} + v_{i,j+1} - \alpha_0 (q_{i,j+1} - q_{i,j})\right)
	\end{align}
	and $\alpha_0 = 0$.  This choice reflects the fact that the presence of numerical diffusion would lead to inaccurate solutions for small values of the scaling parameter $\varepsilon$. The main reason is that such terms are of order $O(\varepsilon^{-2})$, and therefore the effect of numerical diffusion would also scale as $O(\varepsilon^{-2})$.
	
	\item The divergence of $u$  in the Helmholtz equation \eqref{step2} is discretized as 
	\begin{align}
		D_0 \cdot u = \frac{1}{2 \Delta x} (u_{i+1,j}- u_{i-1,j}) +\frac{1}{2 \Delta y} (u_{i,j+1}- u_{i,j-1}).
	\end{align}
	The second-order derivative $\nabla_x \cdot \nabla_x \cdot B(v)$ in \eqref{step2} is discretized in the following way
	\begin{align}
		D_0^2 : B(v) &= -D_{xx} v_1+ 2D_{xy} v_2 + D_{yy} v_1
		= -\frac{1}{\Delta x^2} \left( {v_1}_{i+1,j} - 2{v_1}_{i,j}+ {v_1}_{i-1,j}\right) \nonumber
		\\ &+ \frac{1}{2 \Delta x \Delta y}\left( \left( {v_2}_{i+1,j+1} - {v_2}_{i-1,j+1}\right) - \left( {v_2}_{i+1,j-1} - {v_2}_{i-1,j-1}\right)\right) \nonumber
		\\ &+ \frac{1}{\Delta y^2} \left( {v_1}_{i,j+1} - 2{v_1}_{i,j}+ {v_1}_{i,j-1}\right)
	\end{align}
	while the Laplacian operator $\Delta_x \theta$ is given by
	\begin{align}
		D_{\Delta} \theta = \frac{{\theta}_{i+1,j} - 2{\theta}_{i,j}+ {\theta}_{i-1,j}}{\Delta x^2} + 
		\frac{{\theta}_{i,j+1} - 2{\theta}_{i,j}+ {\theta}_{i,j-1}}{\Delta y^2}.
	\end{align}
	\item Finally, the spatial discretization of $\nabla_x \theta$ in \eqref{step3} is also given by a central difference approximation $D_0 \ \theta$
		\begin{align}
		D_{0}\ \theta = \left(\frac{1}{2\Delta x} (\theta_{i+1,j} - \theta_{i-1,j}) , \frac{1}{2\Delta y} (\theta_{i,j+1} - \theta_{i,j-1})\right)
	\end{align}
	
\end{itemize}
In summary, the fully discretized family of numerical schemes, based on the IMEX–RK approach and on the finite–difference approximations introduced above, can be written as follows:
\begin{align}\nonumber
	&v^{(i)} = \frac{\tau \varepsilon^2}{\tau \varepsilon^2 + \Delta t a_{ii}}v^n -  \frac{\Delta t}{\tau \varepsilon^2 + \Delta t a_{ii}}\left[ \sum_{j=1}^{i-1} \tilde{a}_{ij} \left( \frac{\tau }{4} D_{LF}^0 \cdot B(u^{(j)}) + \tau \varepsilon^2 D_{LF} \ q^{(j)} -  F(u^{(j)})\right)\right]\\ &- \frac{\Delta t}{\tau \varepsilon^2 + \Delta t a_{ii}} \sum_{j=1}^{i-1} a_{ij}  v^{(j)} \\ \nonumber
	&D_{\Delta} \theta^{(i)} - \frac{2 \varepsilon^2}{\Delta t^2 a_{ii}^2} \theta^{(i)} = - \frac{2 \varepsilon^2}{\Delta t^2 a_{ii}^2} \theta^n + \frac{1}{\Delta t a_{ii}^2}\left[\sum_{j=1}^{i-1} a_{ij} D_0 \cdot u^{(j)} \right] \\&+ \frac{1}{\Delta t a_{ii}} D_0 \cdot u^{n} - \frac{1}{ a_{ii}} \sum_{j=1}^{i} a_{ij} D_0^2: B(v^{(j)}) + \frac{1}{ a_{ii}} \sum_{j=1}^{i-1} a_{ij} D_{\Delta}  \theta^{(j)} \\
	&u^{(i)} = u^n - \Delta t \left[\sum_{j=1}^i a_{ij} D_{LF} \cdot B(v^{(j)}) + \sum_{j=1}^i a_{ij} D_0\  \theta^{(j)} \right] \\
	&q^{(i)} = \frac{\tau \varepsilon^2}{\tau \varepsilon^2 + \Delta t a_{ii}}q^n -  \frac{\Delta t}{\tau \varepsilon^2 + \Delta t a_{ii}}\left[ \sum_{j=1}^i a_{ij} \frac{\tau}{2} D_{LF}^0 \cdot v^{(j)} + \sum_{j=1}^{i-1} a_{ij}  q^{(j)}\right] 
\end{align}
The extension to high–order space discretization is achieved by replacing all local Lax–Friedrichs flux functions with high order shock capturing schemes based on WENO reconstruction \cite{Shu}, while all central difference approximations are modified accordingly to attain fourth order accuracy.
We omit the details for simplicity.

\section{Numerical tests}\label{sec4}
In this section, we present a set of numerical experiments designed to assess the performance of the proposed schemes. The tests focus on classical benchmark problems for incompressible flows in two dimensions and aim to verify accuracy, stability, and robustness across different regimes.

\subsection{Thick shear layer test case in 2D}
We first consider the following periodic two-dimensional model problem. Let $(x,y) \in \left[ 0,2 \pi\right]^2$ and let the initial flow to consist of a horizontal shear-layer of finite thickness, perturbed by a small amplitude vertical velocity of the form 
\begin{align}\label{InitialVelocity}
	u_1(x,y,0) = \begin{cases}\tanh(\frac{1}{\rho}(y-\pi/2)), & y \leq \pi \\ \tanh(\frac{1}{\rho}(3\pi/2 - y)), & y> \pi \end{cases}
\end{align}
and
\begin{align}
	u_2(x,y,0) = \delta \sin(x)
\end{align}
with the perturbation parameter $\delta = 0.05$ and the "thick" shear-layer width parameter $\rho = \frac{\pi}{15}$, see \cite{Bell} for details. A convergence study, using the $L_1$ and $L_\infty$ norms, for the vorticity variable is shown in Table 1 and Table 2 for the following first and second order IMEX-RK methods 
\begin{itemize}
	\item First order GSA Euler IMEX scheme
	\begin{align}
			\text{Explicit:  } \begin{array}{c|c}
					\begin{array}{c}
							0 \\  1
						\end{array} &  
					\begin{array}{cc}
							0 & 0  \\ 1 & 0  
						\end{array} \\ \hline   & \begin{array}{cc}
							1 &  0
						\end{array}
				\end{array} \qquad \text{Implicit:  } \begin{array}{c|c}
					\begin{array}{c}
							0  \\ 1
						\end{array} &  
					\begin{array}{cc}
							0 & 0 \\ 0 & 1  
						\end{array} \\ \hline   & \begin{array}{cc}
							0 & 1 
						\end{array}
				\end{array} 
		\end{align}
	
	\item Second order GSA IMEX RK scheme
	\begin{align}
			\text{Explicit:  } \begin{array}{c|c}
					\begin{array}{c}
							0 \\ c \\ 1
						\end{array} &  
					\begin{array}{ccc}
							0 & 0 & 0 \\ c & 0 & 0 \\ 1-1/(2c) & 1/(2c) & 0
						\end{array} \\ \hline   & \begin{array}{ccc}
							1-1/(2c) & 1/(2c) & 0
						\end{array}
				\end{array} \qquad \text{Implicit:  } \begin{array}{c|c}
					\begin{array}{c}
							0 \\ c \\ 1
						\end{array} &  
					\begin{array}{ccc}
							0 & 0 & 0 \\ 0 & c  & 0 \\ 0 & 1-\gamma & \gamma
						\end{array} \\ \hline   & \begin{array}{ccc}
							0 & 1-\gamma & \gamma
						\end{array}
				\end{array} 
		\end{align}
	with $\gamma = (c-1/2)/(c-1)$ and $c = 2.25$, cf. \cite{Bosc},
\end{itemize}
together with a third order WENO reconstruction for the space discretization. The reference solution is obtained from the numerical approximation computed on a fine mesh of $513 \times 513$ grid points. The results reported in Table 1 and Table 2 confirm that the two methods achieve the expected order of convergence. In Figure~\ref{fig:1}, we investigate the uniform convergence for different values of $\varepsilon$ and two distinct values of the diffusion constant $\tau$, namely $\tau = 0$ on the left and $\tau = 0.05$ on the right. These results also demonstrate uniform accuracy across varying values of the scaling parameter, provided that the other discretization parameters are fixed. Figure~\ref{fig:2} displays the vorticity profiles, which are consistent with those produced by other high-order schemes for the incompressible Euler equations, see for example~\cite{Bell,Klar99, BOSCARINO} and the references therein. The right panel of the same figure reports the time evolution of the $L^\infty$ norm of the divergence ${u_1}_x + {u_2}_y$. We observe that the numerical divergence of the velocity grows over time, generating increasingly fine-scale structures in the solution. This effect is due to the choice of the space discretization which do not guarantee at the numerical level the preservation fo the divergence-free condition and it is also also reported and documented in~\cite{BOSCARINO}.
In Figure~\ref{fig:3}, we again observe that the proposed schemes produce vorticity fields comparable to those obtained with other high-order methods for the incompressible Navier–Stokes equations when $\tau > 0$ (see \cite{Bell, BOSCARINO, Klar99}). In this case, the divergence of the velocity grows more slowly than in the incompressible Euler setting, owing to the presence of smoother vortical structures. This spurious effect is due, as already stated, to the choice of the space discretization and we plan to address it in future investigations.
Figure~\ref{fig:3_2} shows that that using larger values of $\varepsilon$ results in smearing out the vortical structures, as expected.
\begin{table}[H]
	\centering
	\begin{tabular}{ |c|c|c|c|c|c|c| } 
		\hline
		N & $L_1$ error  & $L_1$ order & $L_2$ error & $L_2$ order & $L_\infty$ error & $L_\infty$ order\\
		\hline \hline
		33 & 0.0183 & -- & 0.0043 & -- & 0.0021 & --\\ 
		65 & 0.0123 & 0.5734 & 0.0031 & 0.4522 & 0.0016 & 0.4105  \\ 
		129  & 0.0069 & 0.8310 & 0.0019 & 0.6964 & 0.0010 & 0.6400 \\
		257  & 0.0028 & 1.2867 & 8.423e-04 & 1.1904 & 4.599e-04 &  1.1212 \\
		
		\hline
	\end{tabular}
	\caption{ $L_1$, $L_2$ and $L_\infty$ norms of the error and corresponding orders of convergence for the vorticity variable for the thick double shear layer test case with $\varepsilon = 10^{-6}$ and $\tau = 0$ at time $t =1$ for the first order IMEX method.  }
\end{table}

\begin{table}[H]
	\centering
	\begin{tabular}{ |c|c|c|c|c|c|c| } 
		\hline
		N & $L_1$ error  & $L_1$ order & $L_2$ error & $L_2$ order & $L_\infty$ error & $L_\infty$ order\\
		\hline \hline
		33 & 0.0111 & -- & 0.0039 & -- & 0.0019 & --\\ 
		65 & 0.0039 & 1.5141 & 0.0012 & 1.4447 & 7.4941e-04 & 1.3086  \\ 129  & 9.0288e-04 & 2.1069 & 2.7192e-04 & 2.1036 & 1.7520e-04 & 2.0967 \\
		257  & 2.2279e-04 & 2.0189 & 7.0318e-05 & 1.9512 & 4.0079e-05 &  2.1281\\
		
		\hline
	\end{tabular}
	\caption{$L_1$, $L_2$ and $L_\infty$ norms of the error and corresponding orders of convergence for the vorticity variable for the thick double shear layer test case with $\varepsilon = 10^{-6}$ and $\tau = 0$ at time $t =1$ for the second order IMEX method.  }
\end{table}

\begin{figure}[H]
	\centering
	\begin{subfigure}{0.45\textwidth}
		\centering
		\includegraphics[width=\linewidth]{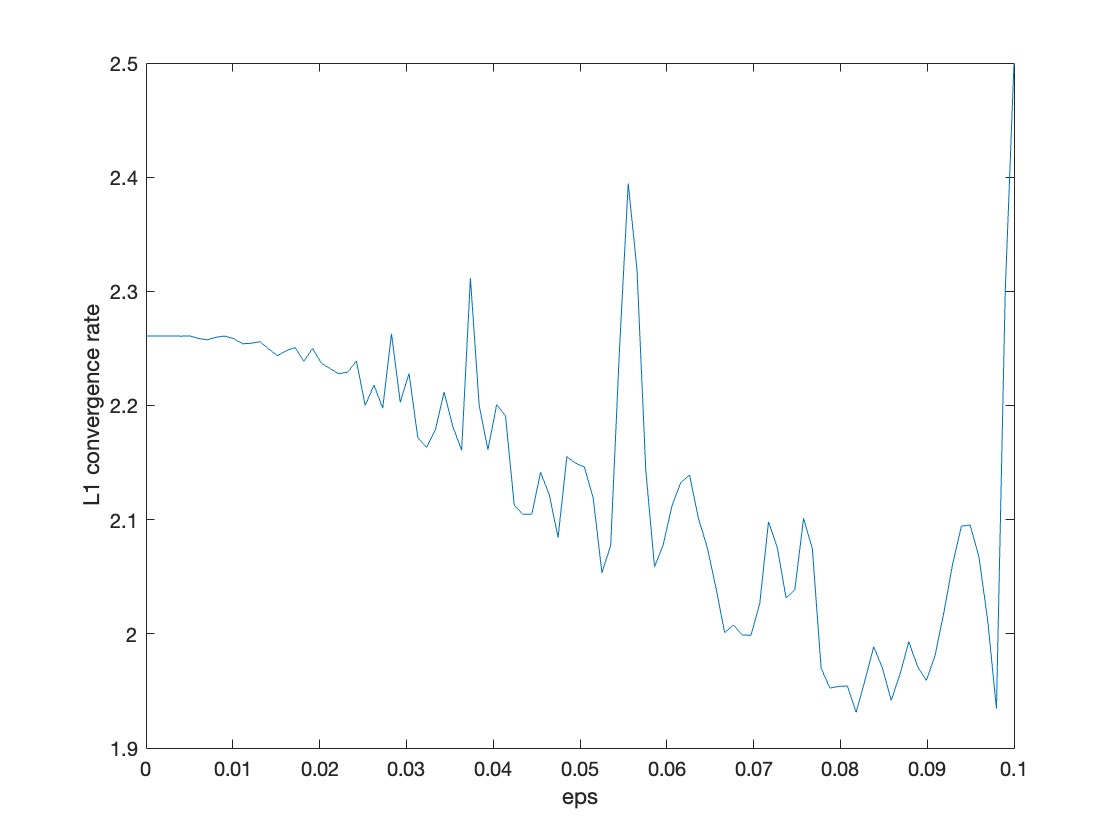}
	\end{subfigure}
	\centering
	\begin{subfigure}{0.45\textwidth}
		\centering
		\includegraphics[width=\linewidth]{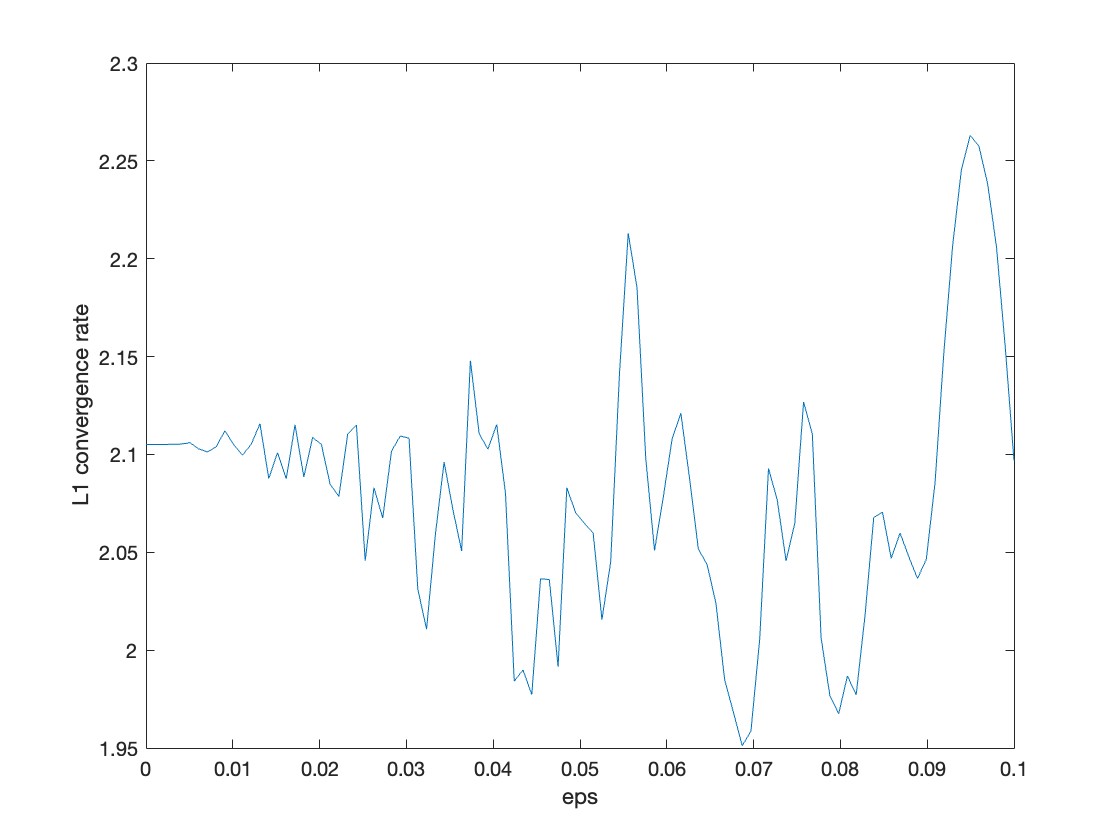}
	\end{subfigure}
	\caption{Convergence rate for the second order IMEX-RK method in time and third order WENO reconstruction in space for different values of $\varepsilon$. (left panel: $\tau = 0$, right panel: $\tau = 0.05$)}
	\label{fig:1}
\end{figure}

\begin{figure}[h]
	\centering
	\begin{subfigure}{0.4\textwidth}
		\centering
		\includegraphics[width=\linewidth]{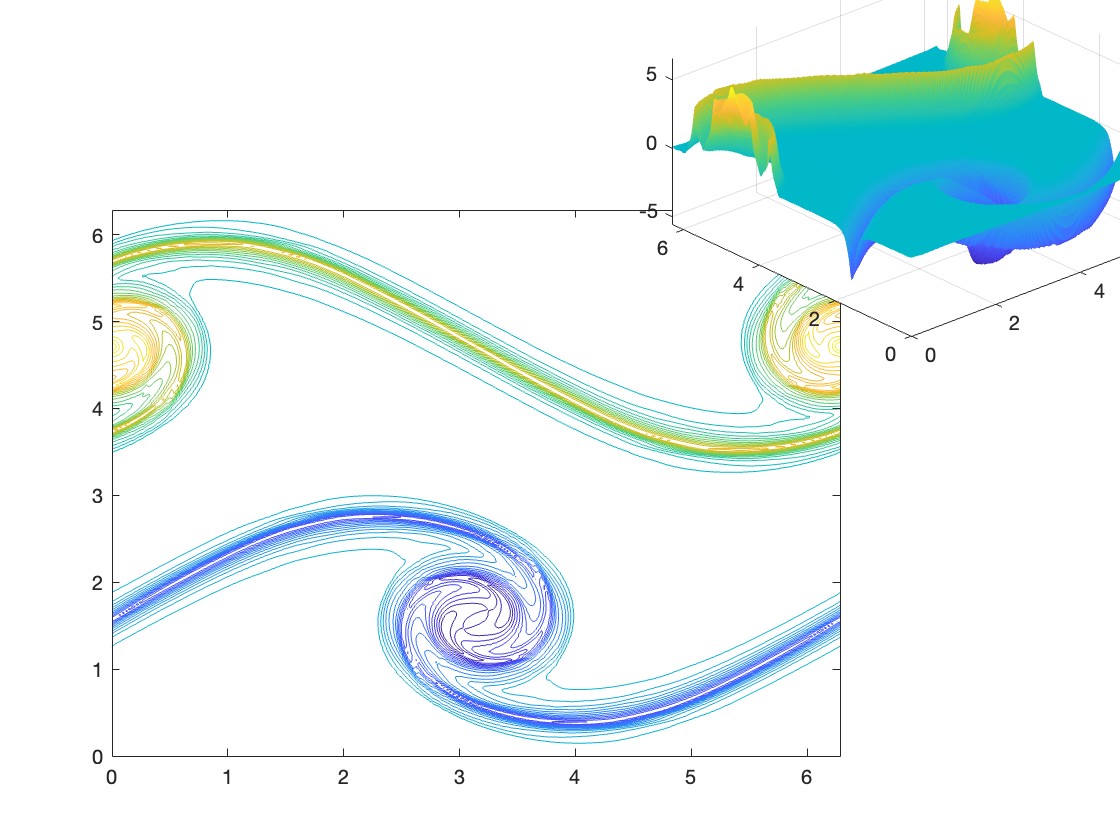}
	\end{subfigure}
%	\begin{subfigure}{0.4\textwidth}
%		\centering
%		\includegraphics[width=\linewidth]{Images/Div_257_tau0_RK2_eps1e-6}
%	\end{subfigure}
	\begin{subfigure}{0.4\textwidth}
		\centering
		\includegraphics[width=\linewidth]{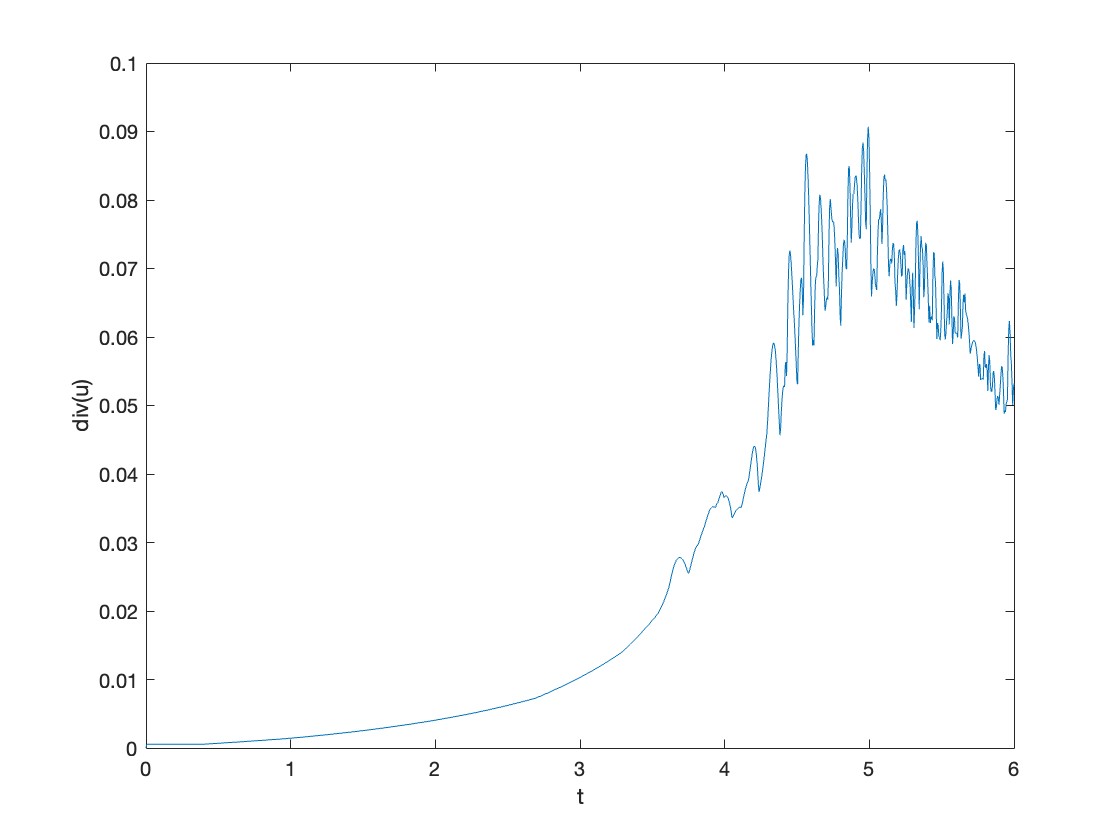}
	\end{subfigure}
%	\caption{Thick shear layer test case in 2D with $\tau = 0$ and $\varepsilon = 10^{-6}$. The vorticity (upper panel) and the divergence of the velocity (lower left panel) are shown at $t = 6$. The lower right panel is the time history of the $L^\infty$ norm of the divergence.}
	\caption{Thick shear layer test case in 2D with $\tau = 0$ (Incompressible Euler case) and $\varepsilon = 10^{-6}$. The vorticity (left panel) and the time history of the $L^\infty$ norm of the divergence (right panel) are shown at $t = 6$.}
	\label{fig:2}
\end{figure}

\begin{figure}[h]
	\centering
	\begin{subfigure}{0.4\textwidth}
		\centering
		\includegraphics[width=\linewidth]{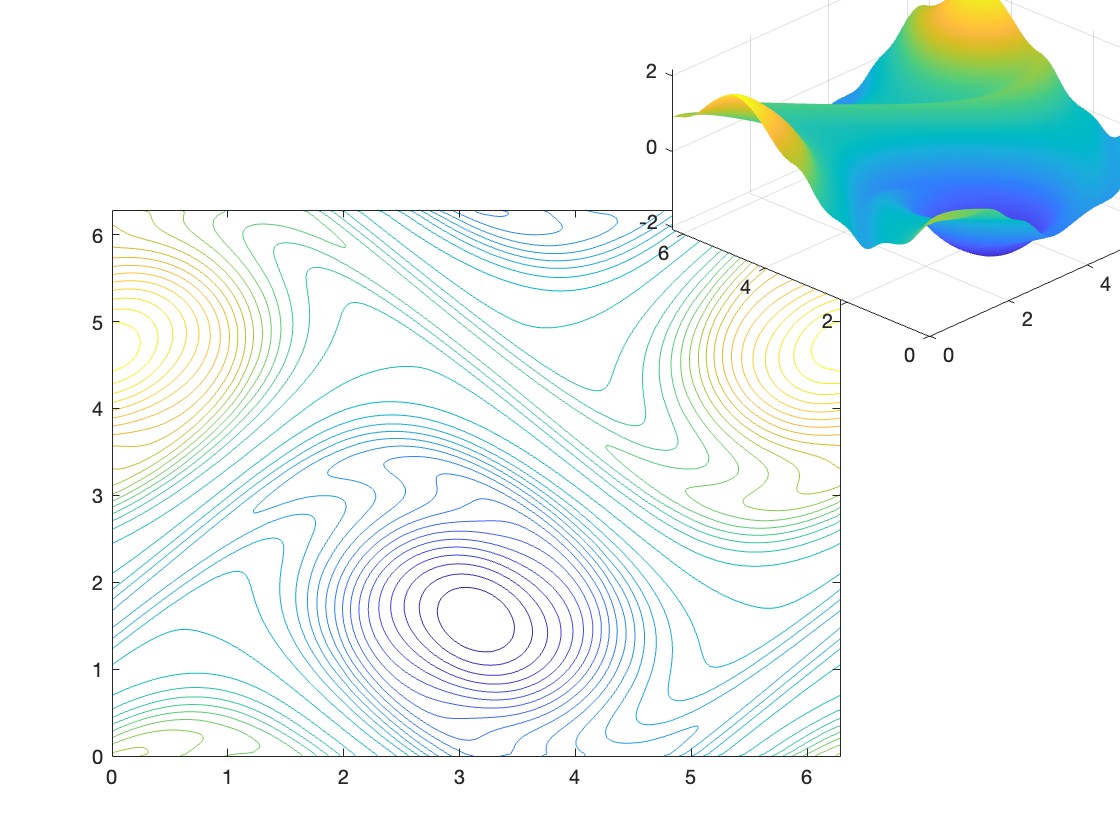}
	\end{subfigure}
%	\begin{subfigure}{0.4\textwidth}
%		\centering
%		\includegraphics[width=\linewidth]{Images/Div_257_tau0.05_cont3}
%	\end{subfigure}
	\begin{subfigure}{0.4\textwidth}
		\centering
		\includegraphics[width=\linewidth]{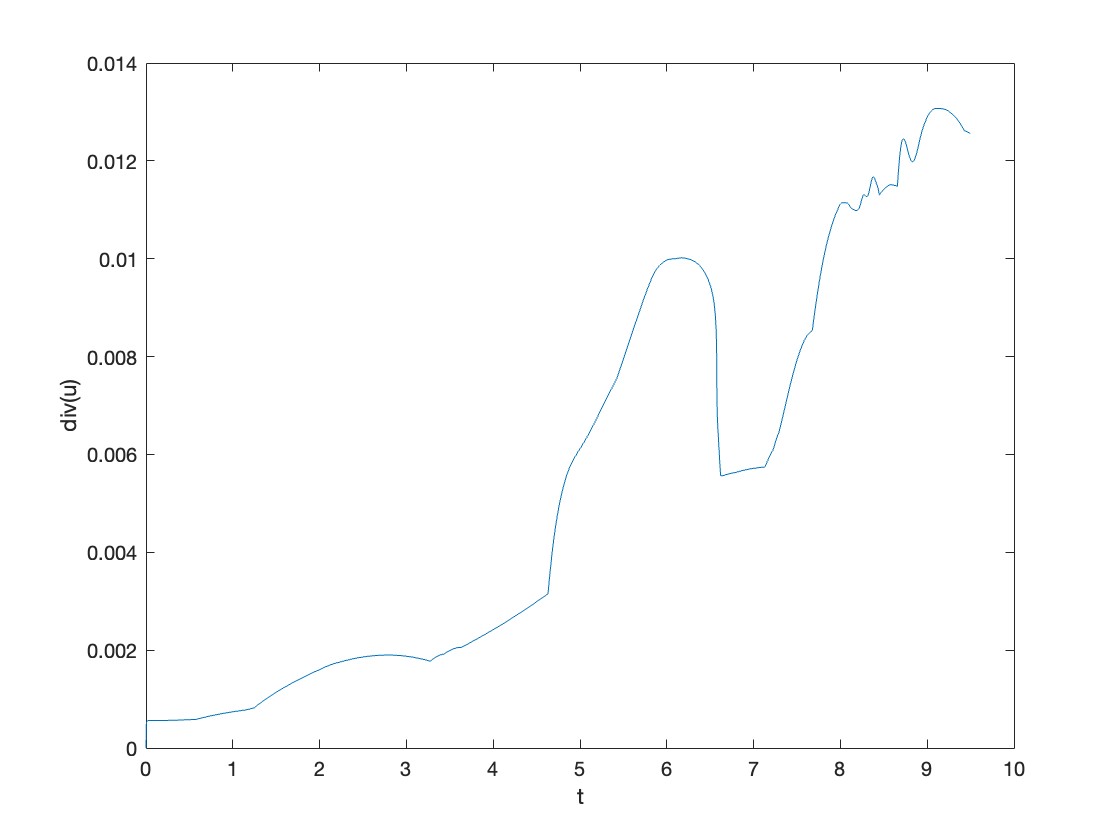}
	\end{subfigure}
%	\caption{Thick shear layer test case in 2D with $\tau = 0.05$ and $\varepsilon = 10^{-6}$. The vorticity (upper panel) and the divergence of the velocity (lower left panel) are shown at $t = 9.5$. The lower right panel is the time history of the $L^\infty$ norm of the divergence.}
		\caption{Thick shear layer test case in 2D with $\tau = 0.05$ (Incompressible Navier-Stokes case) and $\varepsilon = 10^{-6}$. The vorticity (left panel) and the time history of the $L^\infty$ norm of the divergence (right panel) are shown at $t = 9.5$.}
	\label{fig:3}
\end{figure}

\begin{figure}[h]
	\centering
	\begin{subfigure}{0.4\textwidth}
		\centering
		\includegraphics[width=\linewidth]{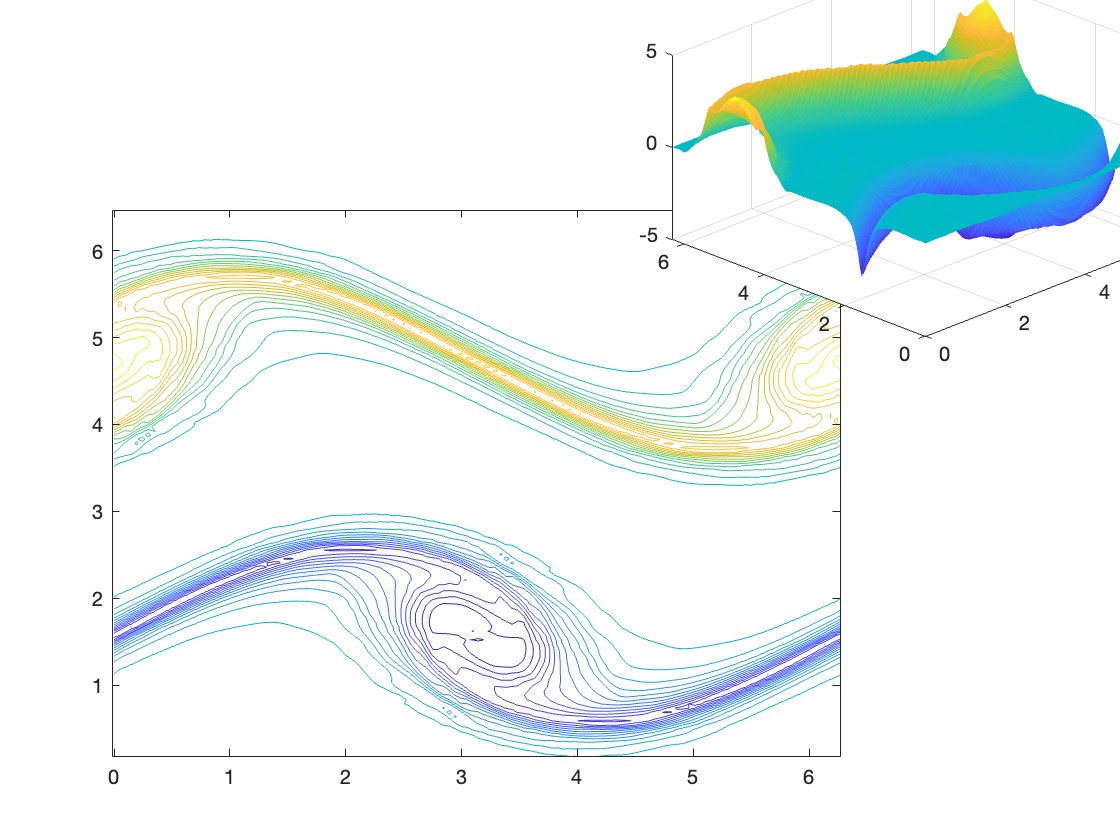}
	\end{subfigure}
%	\begin{subfigure}{0.4\textwidth}
%		\centering
%		\includegraphics[width=\linewidth]{Images/Div_257_tau0.05_cont3}
%	\end{subfigure}
	\begin{subfigure}{0.4\textwidth}
		\centering
		\includegraphics[width=\linewidth]{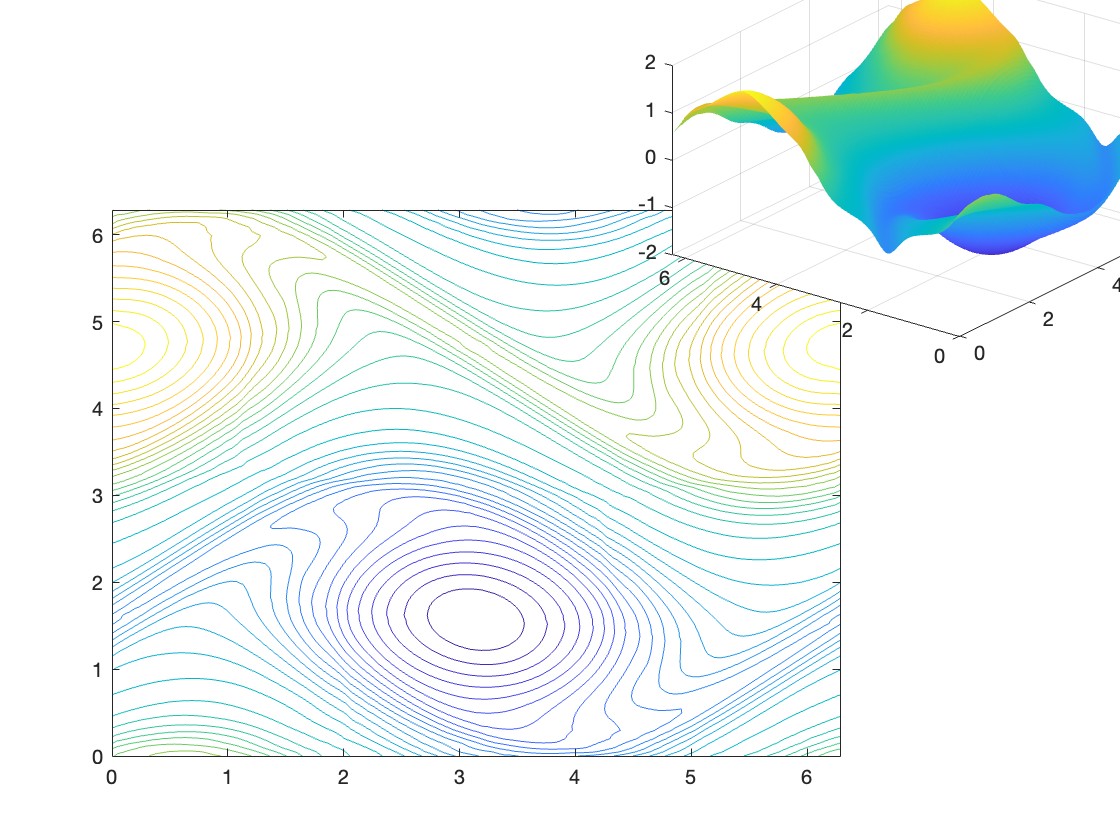}
	\end{subfigure}
%	\caption{Thick shear layer test case in 2D with $\tau = 0.05$ and $\varepsilon = 10^{-6}$. The vorticity (upper panel) and the divergence of the velocity (lower left panel) are shown at $t = 9.5$. The lower right panel is the time history of the $L^\infty$ norm of the divergence.}
		\caption{Thick shear layer test case in 2D with $\varepsilon = 0.25$. The incompressible Euler case with $\tau = 0$ (left panel) is shown at $t = 6$ and the INS case with $\tau = 0.05$ (right panel) is shown at $t = 9.5$.}
	\label{fig:3_2}
\end{figure}

\subsection{Thin shear layer test case in 2D}
For the thin shear layer test case \cite{Bell,Tavelli}, we adopt the same initial data 
$(u(x,y,0))$ as in the thick shear layer set-up, with the only difference that the shear-layer width parameter is set to $\rho = \tfrac{\pi}{50}$. We present the results of the proposed method for $\varepsilon = 10^{-6}$ in Figures \ref{fig:4} and \ref{fig:5}, in terms of vorticity.
More specifically, Figure \ref{fig:4} shows the incompressible Euler case, while Figure \ref{fig:5} corresponds to the incompressible Navier–Stokes case.
Again, the results are in line with those obtained by high-order methods for the limit model \cite{Tavelli}.
As in the thick-layer case, the $L^\infty$-norm of the velocity divergence increases along the vortices due to the choice of the spatial discretization, while the scheme still captures increasingly fine-scale structures.
\begin{figure}[h]
	\centering
	\begin{subfigure}{0.4\textwidth}
		\centering
		\includegraphics[width=\linewidth]{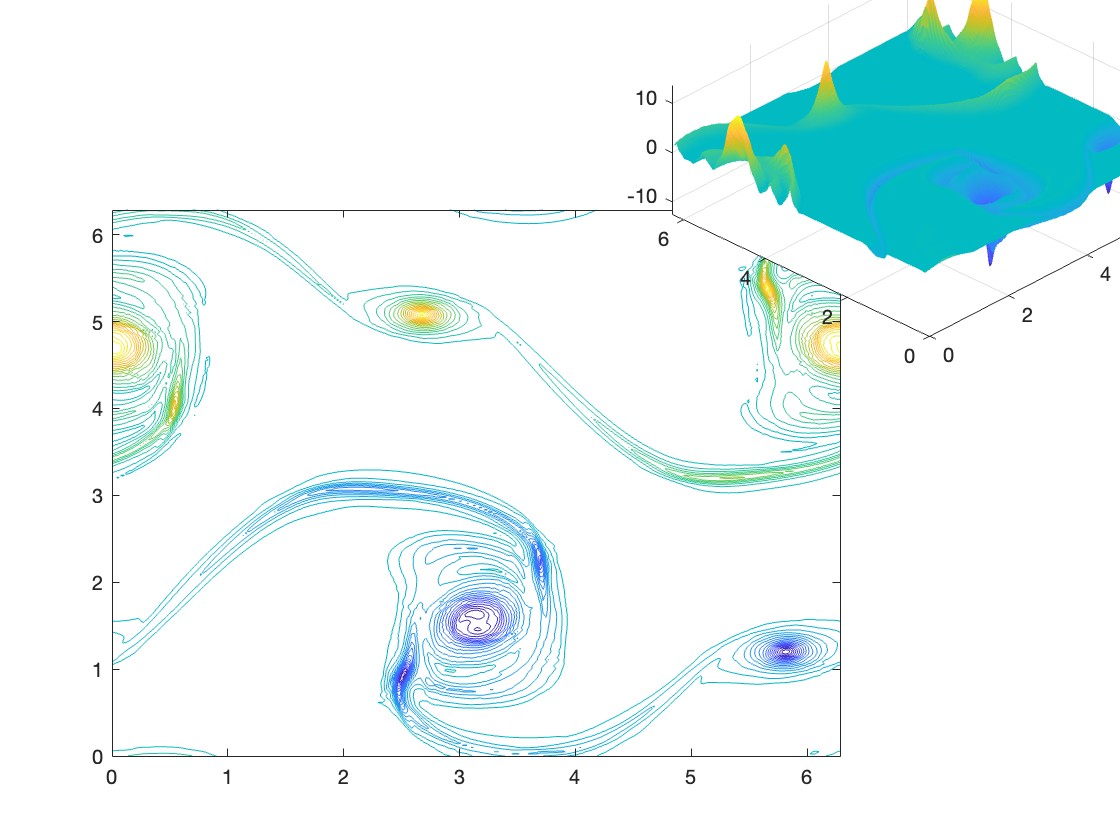}
	\end{subfigure}
%	\begin{subfigure}{0.4\textwidth}
%		\centering
%		\includegraphics[width=\linewidth]{Images/Thin_0_Div}
%	\end{subfigure}
	\begin{subfigure}{0.4\textwidth}
		\centering
		\includegraphics[width=\linewidth]{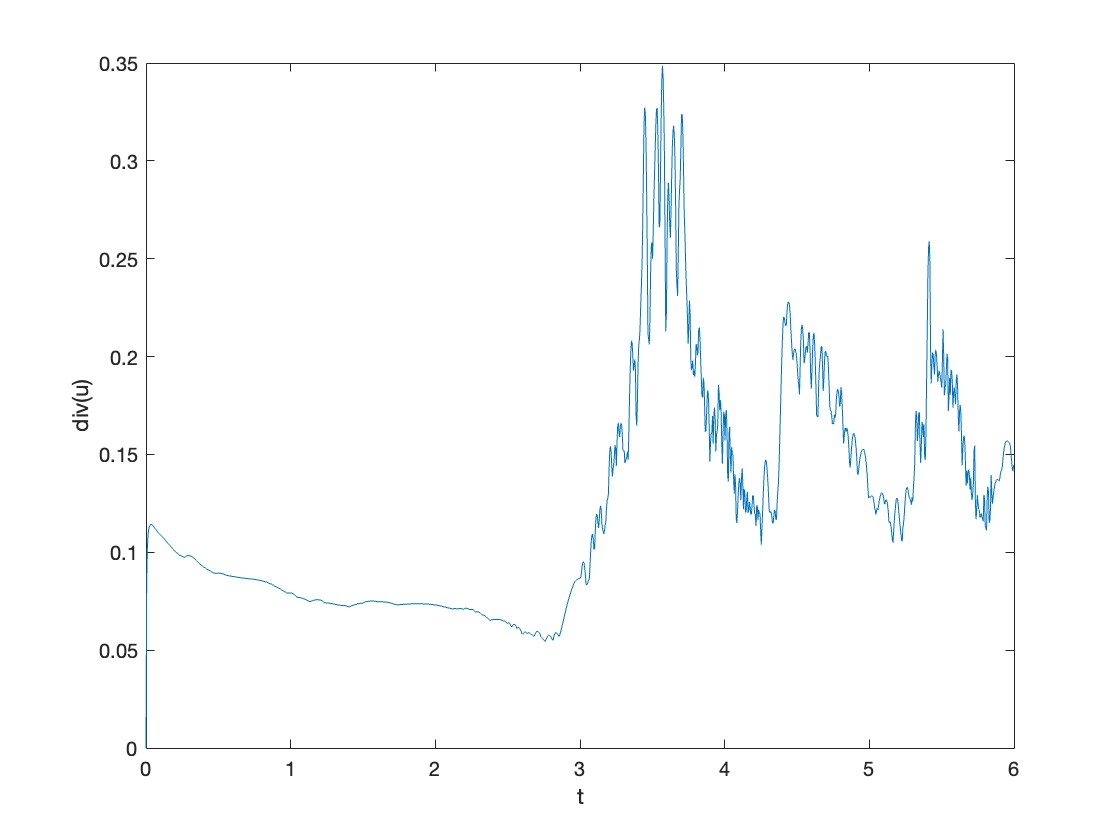}
	\end{subfigure}
%	\caption{Thin shear layer test case in 2D with $\tau = 0$ and $\varepsilon = 10^{-6}$. The vorticity (upper panel) and the divergence of the velocity (lower left panel) are shown at $t = 6$. The lower right panel is the time history of the $L^\infty$ norm of the divergence.}
	\caption{Thin shear layer test case in 2D with $\tau = 0$ (Incompressible Euler case) and $\varepsilon = 10^{-6}$. The vorticity (left panel) and the time history of the $L^\infty$ norm of the divergence (right panel) are shown at $t = 6$.}
	\label{fig:4}
\end{figure}

\subsection{Kevin Helmholtz instability}
Finally, as a last test, we consider the Kevin Helmholtz instability problem \cite{Bell,Klar99} on $\left[ 0,4 \pi\right]^2$ with initial velocity components
\begin{align}
	u_1(x,y,0) = \cos(y)
\end{align}
and
\begin{align}
	u_2(x,y,0) = 0.03 \sin(0.5x).
\end{align}
We compute the numerical solution for $\tau = 0$, namely the incompressible Euler limit, on a $256 \times 256)$ mesh grid up to time $t = 45$. The vorticity field and the corresponding $L^{\infty}$-norm for the divergence of the velocity are shown in Figure \ref{fig:6}. The results are once again in good agreement with those reported in the literature showing that the proposed method is able to keep the structure of the solution for long times. However, we still observe that the velocity divergence slightly increases over time, as finer and finer vorticity structures develop.

\begin{figure}[h]
	\centering
	\begin{subfigure}{0.4\textwidth}
		\centering
		\includegraphics[width=\linewidth]{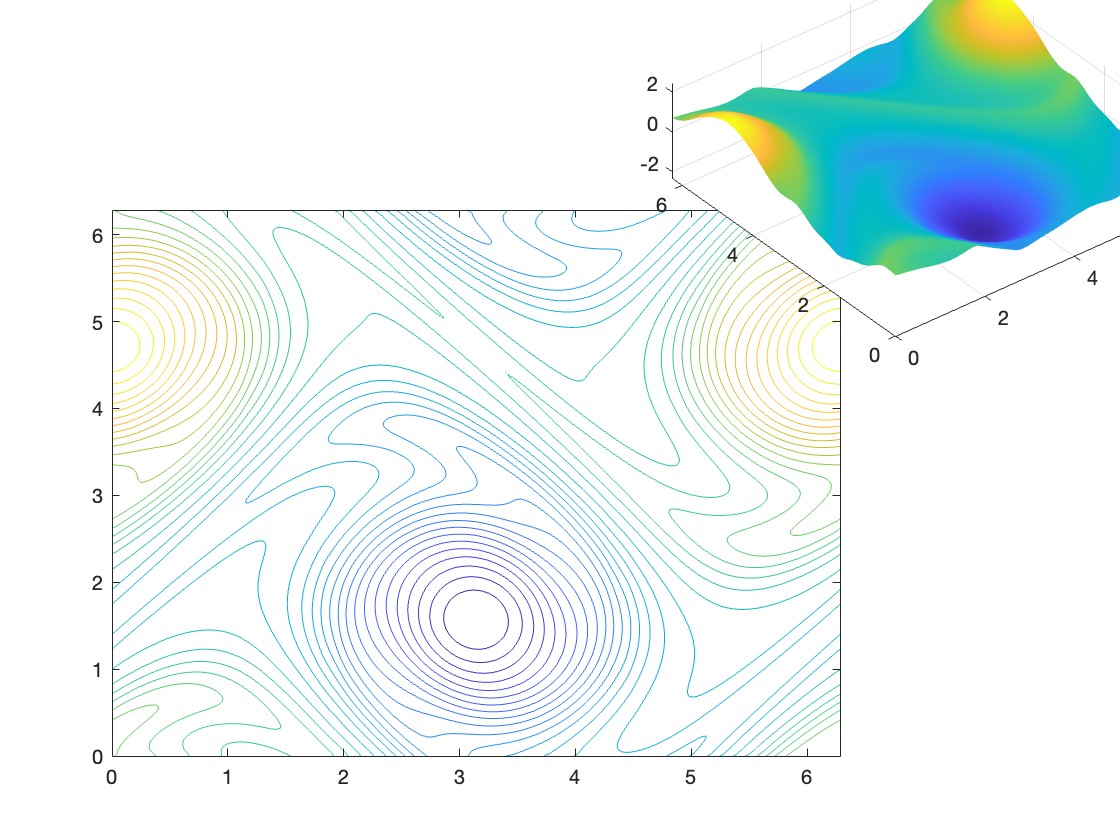}
	\end{subfigure}
%	\begin{subfigure}{0.4\textwidth}
%		\centering
%		\includegraphics[width=\linewidth]{Images/Thin_0.05_Div}
%	\end{subfigure}
	\begin{subfigure}{0.4\textwidth}
		\centering
		\includegraphics[width=\linewidth]{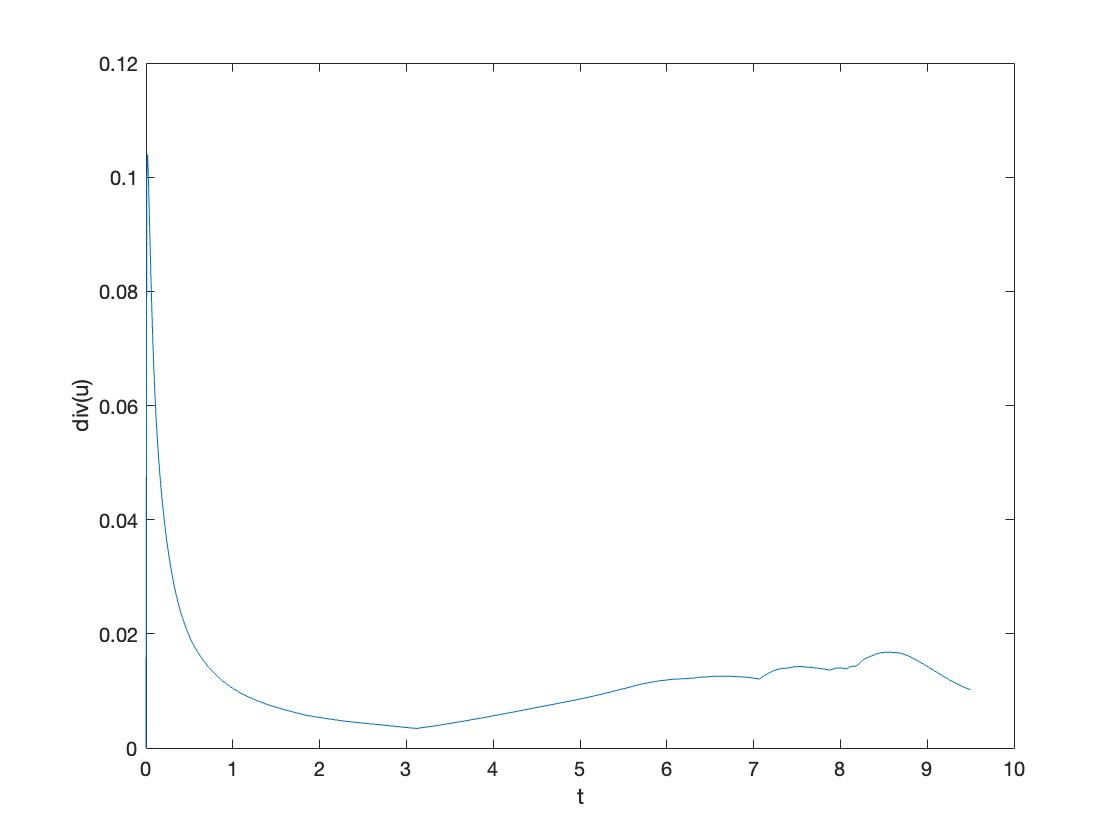}
	\end{subfigure}
%	\caption{Thin shear layer test case in 2D with $\tau = 0.05$ and $\varepsilon = 10^{-6}$. The vorticity (upper panel) and the divergence of the velocity (lower left panel) are shown at $t = 9.5$. The lower right panel is the time history of the $L^\infty$ norm of the divergence.}
	\caption{Thin shear layer test case in 2D with $\tau = 0.05$ (Incompressible Navier-Stokes case) and $\varepsilon = 10^{-6}$. The vorticity (left panel) and the time history of the $L^\infty$ norm of the divergence (right panel) are shown at $t = 9.5$.}
	\label{fig:5}
\end{figure}

\begin{figure}[h]
	\centering
	\begin{subfigure}{0.4\textwidth}
		\centering
		\includegraphics[width=\linewidth]{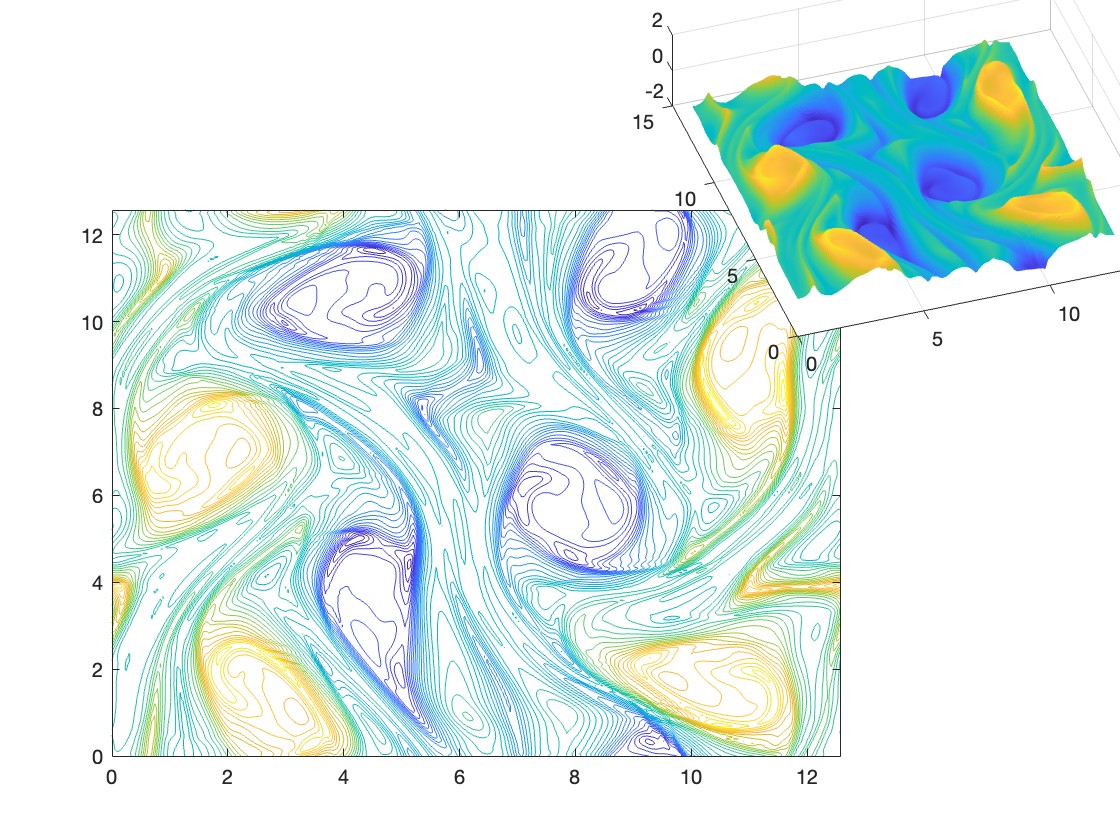}
	\end{subfigure}
%	\begin{subfigure}{0.4\textwidth}
%		\centering
%		\includegraphics[width=\linewidth]{Images/Div_plot_257_KH}
%	\end{subfigure}
	\begin{subfigure}{0.4\textwidth}
		\centering
		\includegraphics[width=\linewidth]{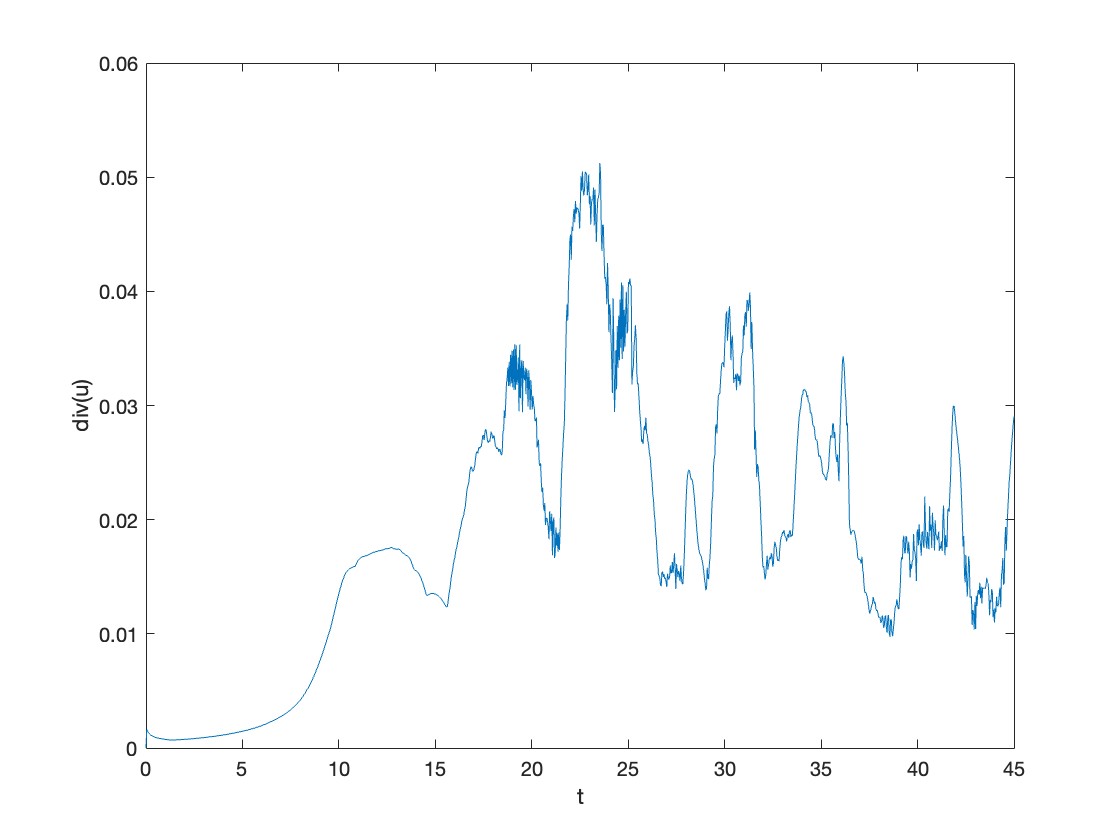}
	\end{subfigure}
%	\caption{Kevin Helmholtz instability with $\tau = 0$ and $\varepsilon = 10^{-6}$. The vorticity (upper panel) and the divergence of the velocity (lower left panel) are shown at $t = 45$. The lower right panel is the time history of the $L^\infty$ norm of the divergence.}
	\caption{Kevin Helmholtz instability with $\tau = 0.05$ and $\varepsilon = 10^{-6}$.(Incompressible Navier-Stokes case). The vorticity (left panel) and the time history of the $L^\infty$ norm of the divergence (right panel) are shown at $t = 45$.}
	\label{fig:6}
\end{figure}

\section{Conclusions}\label{sec5}

%In this paper, we considered a lattice Boltzmann–type model with six discrete velocities under the low Mach scaling and introduced a class of numerical schemes that remain uniformly valid with respect to the scaling parameter, up to the incompressible Navier–Stokes limit. The proposed schemes combine IMEX Runge–Kutta time discretizations with high–order non-oscillatory spatial discretizations within a finite–difference framework. In particular, WENO reconstruction is employed together with high–order central discretizations. In the asymptotic limit $\varepsilon = 0$, the schemes are shown to be consistent with a high–order projection method for the incompressible Navier–Stokes equations, without requiring a coupling between the time step and the scaling parameter. Numerical experiments confirm both the effectiveness and the accuracy of the proposed method. The extension of this approach to physically more relevant systems will be the subject of future research.

In this paper, we have considered a lattice Boltzmann–type model with six discrete velocities under the low-Mach scaling and introduced a class of numerical schemes that remain uniformly valid with respect to the scaling parameter, up to the incompressible Euler or Navier–Stokes limit. The proposed schemes combine IMEX Runge–Kutta time discretizations with high-order non-oscillatory spatial discretizations within a finite-difference framework. In particular, WENO reconstructions are employed together with high-order central schemes. In the asymptotic limit $\varepsilon = 0$, the schemes are shown to be consistent with a high-order projection method for the incompressible Euler and Navier–Stokes equations, without requiring any coupling between the time step and the scaling parameter. Numerical experiments confirm both the effectiveness and the accuracy of the proposed method. A relevant direction for future research will be the extension of this approach to physically more relevant systems such as the Boltzmann equation. One important additional direction of research will consist in the design of spatial discretizations that, analogously to the temporal ones, ensure the exact preservation of the divergence-free condition $\nabla \cdot u = 0$ at all times.

%\newpage
\bibliographystyle{elsarticle-num}

\bibliography{reference_last}

\section*{Appendix A}
In this Appendix, we report the details of the proof for Theorem 3.5 related to IMEX CK or ARS schemes. 
\begin{proof}
	Due to the structure of the IMEX method of type CK or ARS, the first internal stages are equal to the variables at timestep n:
	\begin{align}
	u^{(1)} = u^n, \quad v^{(1)} = v^n, \quad \theta^{(1)} = \theta^n,\quad q^{(1)} = q^n.
	\end{align}
	By rewriting the equation of the internal stages $v^{(i)}$, we obtain
	\begin{align}
		v^{(i)} &= \frac{\tau \varepsilon^2}{\tau \varepsilon^2 + \Delta t a_{ii}}v^n \nonumber \\&- \Delta t \frac{1}{\tau \varepsilon^2 + \Delta t a_{ii}}\left[ \sum_{j=1}^{i-1} \tilde{a}_{ij} \left( \frac{\tau }{4} \nabla_x \cdot B(u^{(j)}) + \tau \varepsilon^2 \nabla_x q^{(j)} -  F(u^{(j)})\right)\right] \label{V_i_A}\\ \nonumber &- \Delta t \sum_{j=1}^{i-1} a_{ij} \frac{1}{\tau \varepsilon^2 + \Delta t a_{ii}} v^{(j)}.
	\end{align}
%	Next, we insert \eqref{V_i} into the last internal stage of $u$.
%	\begin{align}
%		u^{(s)} &= u^n - \Delta t \left[\sum_{j=1}^{s-1} a_{sj} \nabla_x \cdot B(v^{(j)}) + \sum_{j=1}^s a_{sj} \nabla_x  \theta^{(j)} \right] - \Delta t a_{ss}\nabla_x \cdot B(v^{(s)}) 
%	\end{align}
Then, by inserting the above expression into the first equation of \eqref{InternalStages}
and by taking the limit $\varepsilon \rightarrow 0$, we obtain % by inserting the last stage of \eqref{V_i}
	\begin{align}\nonumber
		u^{(i)} &= u^n - \Delta t \left[\sum_{j=1}^{i-1} a_{ij} \nabla_x \cdot B(v^{(j)}) + \sum_{j=1}^i a_{ij} \nabla_x  \theta^{(j)} \right] \\ &- \Delta t \nabla_x \cdot B\left( -\sum_{j=1}^{i-1} \tilde{a}_{ij} \left( \frac{\tau }{4} \nabla_x \cdot B(u^{(j)}) -  F(u^{(j)})\right) - \sum_{j=1}^{i-1} a_{ij} v^{(j)} \right).
	\end{align}
	Rewriting gives
	\begin{align}\nonumber
		u^{(i)} &= u^n - \Delta t \left[\sum_{j=1}^{i-1} a_{ij} \nabla_x \cdot B(v^{(j)}) + \sum_{j=1}^i a_{ij} \nabla_x  \theta^{(j)} \right] \\ &- \Delta t \nabla_x \cdot B\left( -\sum_{j=1}^{i-1} \tilde{a}_{ij} \left( \frac{\tau }{4} \nabla_x \cdot B(u^{(j)}) -  F(u^{(j)})\right) \right)  + \Delta t \sum_{j=1}^{i-1} a_{ij} \nabla_x \cdot B (v^{(j)}) . \label{u_i_5}
	\end{align}
We now consider the stage equations for $\theta$ and $i \geq 1$:
	\begin{align}
		\theta^{(i)} = \theta^n - \Delta t\left[\frac{1}{2 \varepsilon^2}\sum_{j=1}^{i-1} a_{ij} \nabla_x \cdot u^{(j)} \right] - \Delta t \frac{1}{2\varepsilon^2} a_{ii} \nabla_x \cdot u^{(i)}_* + \Delta t^2 \frac{1}{2\varepsilon^2} a_{ii}^2 \Delta_x  \theta^{(i)},
	\end{align}
	where
	\begin{align}
		u^{(i)}_* &= u^n - \Delta t \left[\sum_{j=1}^i a_{ij} \nabla_x \cdot B(v^{(j)}) + \sum_{j=1}^{i-1} a_{ij} \nabla_x  \theta^{(j)} \right].
	\end{align}
	Rewriting gives
	\begin{align}
		\Delta_x \theta^{(i)} - \frac{2\varepsilon^2}{\Delta t^2 a_{ii}^2} \theta^{(i)} = - \frac{2\varepsilon^2}{\Delta t^2 a_{ii}^2} \theta^n + \frac{1}{\Delta t a_{ii}^2} \sum_{j=1}^{i-1} a_{ij} \nabla_x \cdot u^{(j)} +  \frac{1}{\Delta t a_{ii}} \nabla_x \cdot u^{(i)}_*. \label{Theta_i_5}
	\end{align}
Next, we show inductively that all internal stages $u^{(i)}$ fulfil the divergence-free condition in the limit $\varepsilon \rightarrow 0$: For the base case $i= 1$, this is true because $u^{(1)} = u^n$ and the initial velocity field is divergence free for hypothesis. For the inductive step $(i-1) \mapsto i$, \eqref{u_i_5} gives
	 \begin{align}\nonumber
		\nabla_x \cdot u^{(i)} &= \nabla_x \cdot u^n - \Delta t \left[\sum_{j=1}^{i-1} a_{ij} \nabla_x \cdot \nabla_x \cdot B(v^{(j)}) + \sum_{j=1}^{i} a_{ij} \Delta_x  \theta^{(j)} \right] \\ &+ \Delta t \nabla_x \cdot \nabla_x \cdot B\left( \sum_{j=1}^{i-1} \tilde{a}_{ij} \left( \frac{\tau }{4} \nabla_x \cdot B(u^{(j)}) -  F(u^{(j)})\right) \right)  + \Delta t \sum_{j=1}^{i-1} a_{ij}\nabla_x \cdot \nabla_x \cdot B (v^{(j)}) . \label{induction_step_5}
	\end{align}
	Using \eqref{Theta_i_5} and the induction hypothesis $\nabla_x \cdot u^{(i-1)} = 0$ gives 
	\begin{align}
	\sum_{j=1}^{i} a_{ij} \Delta_x  \theta^{(j)} = \sum_{j=1}^{i-1} a_{ij} \Delta_x  \theta^{(j)} + \frac{1}{\Delta t} \left( \nabla_x \cdot u^n - \Delta t \left(\sum_{j=1}^{i}a_{ij}\nabla_x \cdot \nabla_x \cdot B(v^{(j)}) +\sum_{j=1}^{i-1} a_{ij} \Delta_x  \theta^{(j)} \right)\right). \label{sum_lapl_theta_5}
	\end{align}
	On the other hand, for $\varepsilon \rightarrow 0$, equation \eqref{V_i} gives
	\begin{align}
	v^{(j)} = - \frac{1}{a_{jj}} \left( \sum_{k=1}^{j-1} \tilde{a}_{jk} \left(\frac{\tau}{4} \nabla_x \cdot B(u^{(k)}) - F(u^{(k)}) \right)  + \sum_{k=1}^{j-1} a_{jk} v^{(k)}\right). \label{v_j_eps_0_5}
	\end{align}
	Plugging then \eqref{sum_lapl_theta_5} and \eqref{v_j_eps_0_5} into \eqref{induction_step_5}, we obtain in the limit
	\begin{align}
	\nabla_x \cdot u^{(i)} = 0
	\end{align}
	Now, by using the following properties
	\begin{align}
		\nabla_x \cdot B(\nabla_x \cdot B(u)) &= \Delta_x u \\
		\nabla_x \cdot B(F(u)) &= (u\cdot \nabla_x) u - \nabla_x \left( \frac{|u|^2}{2}\right) + u \nabla \cdot u,
	\end{align}
	the divergence–free initial condition, as well as the fact that all internal stages $u^{(i)}$ are divergence-free, we get
	\begin{align}\label{u_s_5}
		u^{(i)} &= u^n - \Delta t \left[ \sum_{j=1}^i a_{ij} \nabla_x  \theta^{(j)} \right] \nonumber \\ &- \Delta t  \sum_{j=1}^{i-1} \tilde{a}_{ij} \left( -\frac{\tau }{4} \Delta_x u^{(j)} +  (u^{(j)}\cdot \nabla_x) u^{(j)} - \nabla_x \left( \frac{|u^{(j)}|^2}{2}\right)\right).  
	\end{align}
Now, due to the fact that we use an GSA IMEX RK scheme, the above relation is enough to state that the numerical solution becomes 
	\begin{align}\nonumber
		u^{n+1} &= u^n - \Delta t \left[ \sum_{i=1}^s w_{i} \nabla_x  \theta^{(i)} \right] \\ &- \Delta t  \sum_{i=1}^{s} \tilde{w}_{i} \left( -\frac{\tau }{4} \Delta_x u^{(i)} +  (u^{(i)}\cdot \nabla_x) u^{(i)} - \nabla_x \left( \frac{|u^{(i)}|^2}{2}\right)\right).  
	\end{align}
Finally, because by definition, we have
	\begin{align}\label{Theta_P_Prop_5}
		\theta^{(i)} = p^{(i)} + \frac{|u^{(i)}|^2}{2}
	\end{align}
and since $\sum_{i=1}^s \tilde{w}_i = \sum_{i=1}^s w_i$, we obtain for the numerical solution of the velocity field $u$ the following expression
	\begin{align}\label{Glg1}
		u^{n+1} &= u^n - \Delta t \left( \sum_{i=1}^{s} \tilde{w}_{i} \left( -\frac{\tau }{4} \Delta_x u^{(i)} +  (u^{(i)}\cdot \nabla_x) u^{(i)} \right)
		+  \sum_{i=1}^s w_{i} \nabla_x  p^{(i)} \right), 
	\end{align}
which is a consistent high order Implicit-Explicit time discretization of equation \eqref{INS1}.
We now consider the rewritten stage equations for $\theta$ \eqref{Theta_i_5}.
	Taking the limit $\varepsilon \rightarrow 0$, by using \eqref{u_s_5} as well as the divergence-free initial condition for $u^n$ and the fact that all internal stages of $u$ are divergence-free in the limit, we get 
	\begin{align}%\nonumber
		\Delta_x \theta^{(i)}  &=   \frac{1}{ a_{ii}} \nabla_x \cdot \left[ \sum_{j=1}^{i-1} \tilde{a}_{ij} \left( \frac{\tau }{4} \Delta_x u^{(j)} -  (u^{(j)}\cdot \nabla_x) u^{(j)} + \nabla_x \left( \frac{|u^{(j)}|^2}{2}\right)\right) \right]
		- \sum_{j=1}^{i-1} a_{ij} \Delta_x  \theta^{(j)}.
	\end{align}
The above equation can be recast as
	\begin{align}
		a_{ii}\Delta_x \theta^{(i)}  &= \nabla_x \cdot \left[ \sum_{j=1}^{i-1} \tilde{a}_{ij} \left( \frac{\tau }{4} \Delta_x u^{(j)} -  (u^{(j)}\cdot \nabla_x) u^{(j)}\right) \right] +\sum_{j=1}^{i-1} \tilde{a}_{ij} \Delta_x \frac{|u^{(j)}|^2}{2}
		- \sum_{j=1}^{i-1} a_{ij} \Delta_x  \theta^{(j)}
	\end{align}
	and due to the fact that a GSA IMEX RK scheme is employed, this is equivalent to
	\begin{align}
		\sum_{i=1}^{s} w_{i} \Delta_x  \theta^{(i)} = \nabla_x \cdot \left[ \sum_{i=1}^{s} \tilde{w}_{i} \left( \frac{\tau }{4} \Delta_x u^{(i)} -  (u^{(i)}\cdot \nabla_x) u^{(i)}\right) \right] +\sum_{i=1}^{s} \tilde{w}_{i} \Delta_x \frac{|u^{(i)}|^2}{2}.
	\end{align}
Thanks to \eqref{Theta_P_Prop_5} and $\sum_{i=1}^s w_i = \sum_{i=1}^s \tilde{w}_i$, we finally get
	\begin{align}\label{Glg2}
		\sum_{i=1}^{s} w_{i} \Delta_x  p^{(i)} = \nabla_x \cdot \left[ \sum_{i=1}^{s} \tilde{w}_{i} \left( \frac{\tau }{4} \Delta_x u^{(i)} -  (u^{(i)}\cdot \nabla_x) u^{(i)}\right) \right],
	\end{align}
which is a consistent high order Implicit-Explicit time discretization of equation \eqref{INS2}.
\end{proof}

\end{document}